\documentclass[11pt]{amsart}
\usepackage{marginnote}
\usepackage[a4paper, centering]{geometry}
\usepackage{cancel}
\usepackage{amsmath}
\usepackage{amsthm}
\usepackage{yhmath}
\allowdisplaybreaks[1]
\usepackage{amssymb,mathrsfs}

\usepackage{array}   
\newcolumntype{L}{>{$}l<{$}} 

\usepackage{tensor}

\usepackage{slashed}

\usepackage{wrapfig}
\usepackage[pdftex]{graphicx}
\usepackage{pdfpages}

\usepackage{lmodern}
 \usepackage[T1]{fontenc}

\usepackage{esint}
\usepackage{epstopdf}
\usepackage{amscd}
\usepackage{enumerate}
\usepackage{fancyhdr} 
\usepackage{microtype}
\usepackage[all]{xy}
\xyoption{poly}

\usepackage[utf8]{inputenc}
\usepackage[pdfencoding=auto,psdextra]{hyperref}
\usepackage[capitalise]{cleveref}
\crefname{lem}{Lemma}{Lemmas}
\Crefname{lem}{Lemma}{Lemmas}
\crefname{thm}{Theorem}{Theorems}
\Crefname{thm}{Theorem}{Theorems}
\crefname{prop}{Proposition}{Propositions}
\Crefname{prop}{Proposition}{Propositions}
\crefname{defn}{Definition}{Definitions}
\Crefname{defn}{Definition}{Definitions}

\usepackage{mymathsty}
 \numberwithin{equation}{section}

	\newcommand*{\op}[1]{{\mathrm{#1}}}
	
	\DeclareMathSymbol{\Gamma}{\mathalpha}{operators}{0}
	\DeclareMathSymbol{\Delta}{\mathalpha}{operators}{1}
	\DeclareMathSymbol{\Theta}{\mathalpha}{operators}{2}
	\DeclareMathSymbol{\Lambda}{\mathalpha}{operators}{3}
	\DeclareMathSymbol{\Xi}{\mathalpha}{operators}{4}
	\DeclareMathSymbol{\Pi}{\mathalpha}{operators}{5}
	\DeclareMathSymbol{\Sigma}{\mathalpha}{operators}{6}
	\DeclareMathSymbol{\Upsilon}{\mathalpha}{operators}{7}
	\DeclareMathSymbol{\Phi}{\mathalpha}{operators}{8}
	\DeclareMathSymbol{\Psi}{\mathalpha}{operators}{9}
	\DeclareMathSymbol{\Omega}{\mathalpha}{operators}{10}

	\renewcommand{\epsilon}{\varepsilon}
	
	\DeclareMathSymbol{\T}{\mathbin}{AMSb}{"54}
	
\usepackage{mathtools}
\newcommand{\defeq}{\vcentcolon=}

\newcommand{\pFq}[2]{\ensuremath{
        {\phantom{\,}}_{#1}F_{#2}
}
    }

\begin{document}
 \author{Yang Liu}
 \address{
 Max Planck Institute for Mathematics,
 Vivatsgasse 7,
 53111 Bonn,
 Germany
 }
 \email{yangliu@mpim-bonn.mpg.de}
 \title[Hypergeometric function and Modular Curvature I.  ] {
    Hypergeometric function and Modular Curvature I. 
Hypergeometric functions in Heat Coefficients 
}


\date{\today} 
\thanks{
} 

\begin{abstract}

We give a new proof of the rearrangement lemma that works for all dimensions
and all heat coefficients in the study of modular geometry on noncommutative
tori. The building blocks of the spectral functions are
landed in a hypergeometric family knowns as Lauricella functions of type $D$. 
 We investigate the differential and recursive relations among the family and
 obtain a full reduction to  Gauss hypergeometric functions. 
As for applications, we give phase one demonstration on
how the hypergeometric features lead to new simplifications in 
computations involved in the modular geometry. 

\end{abstract}

\maketitle

\tableofcontents

\section{Introduction}

A recent breakthrough in the metric aspect of noncommutative geometry is Connes
and Moscovici's study of the modular Gaussian curvature on noncommutative
two tori \cite{MR3194491}. 
The essential contribution is a spectral framework for exploring  the notion of
intrinsic curvature in regard to a family of curved metrics parametrized by
noncommutative coordinates. 
In Connes's spectral triple paradigm, metrics are implemented by geometric
differential operators, such as  Laplacian or Dirac operators. Modeling on
spectral geometry of Riemannian manifolds, the
corresponding local invariants are encoded in the coefficients of the
small time heat trace expansion, see
\cref{eq:small-time-heat-asym-functional}. 
In contrast to using the metric tensors in Riemannian geometry, 
we  introduce a family of metrics via perturbing  a given  Laplacian or
Dirac by noncommutative coordinates. The simplest example is the conformal
change of metric on $\T^2_\theta$ whose perturbation, roughly speaking, is solely
$\Delta \to \Delta_\varphi = e^{ h /2 } \Delta e^{h /2 }$. The
metric coordinate (referred as log-Weyl factor)  
runs over  self-adjoint elements $h = h^* \in  C^\infty(\T^2_\theta)$
in the algebra of smooth coordinate functions.
Modular geometry mainly concerns variational problems on such family of metrics
adapted from Riemannian geometry.
For example, the modular Gaussian curvature is the functional gradient of the
analogue of OPS (Osgood-Phillips-Sarnak) functional on Riemann surfaces
\cite{osgood1988extremals}.

The quantum features of the metrics reside on the  inevitable difficulty 
that the metric coordinates and their derivatives do not commute,
which abolishes many useful rules in calculus.
The remedy initiated in \cite{MR2907006} is a rearrangement process which
compress the
ansatz into some intriguing spectral functions. A toy example is the second
differential of the power function $\nabla^2 k^3$  where 
$k \in  C^\infty(\T^2_\theta)$ invertible and $\nabla$ is a derivation: 
\begin{align}
    \nabla^2 k^3 &= k^2 (\nabla^2 k) + (\nabla k) k^2 +
    k (\nabla k) (\nabla k) +  (\nabla k) k (\nabla k) + (\nabla k) (\nabla k) k
    \nonumber
    \\
    &= k^2 (1+ \mathbf y) (\nabla^2 k) + k (1+\mathbf y^{(1)} +\mathbf y^{(1)}
    \mathbf y^{(2)}) (\nabla k \otimes  \nabla k)
    .
    \label{eq:toyeg-modspecfun}
\end{align}
We first compute the derivation via the Leibniz rule and then
use the conjugation $\mathbf y = k^{-1}(\cdot )k$ to move the zero order
differential to the very left following the notations in
\S\ref{subsec:smooth-funcal-A-tensor-n-times}. 
A more interesting one is the Duhamel's formula for differentiating  the
exponential of a self-adjoint element:
\begin{align*}
    \nabla  e^h =\frac{e^{\mathbf x} -1}{\mathbf x} (\nabla  h), \, \, \, \, 
    \mathbf x \defeq [\cdot ,h],
\end{align*}
in which the generating function of Bernoulli numbers appears.
%

The modular curvature $R_{\Delta_k}$ with respect to the Weyl factor $k$ 
  in \cref{eq:psecal-R_Delta_k}, looked
unconventional at the first glance, is merely a decorated version of
\cref{eq:toyeg-modspecfun}. In other words, it is a combination of the second
derivatives of the metric coordinate, which is an expected property inherited from 
the notion of scalar curvature in Riemannian geometry.
%
The new ingredients are 
 the spectral functions $K_{ \Delta_{k} }$ and  $H_{ \Delta_{k} }$ 
 (cf. \cref{eq:b2cal-KDeltak,eq:b2cal-HDeltak}) arising
from pseudo-differential calculus, which are more sophisticated and intriguing
comparing to truncated  geometry series $1+y$ and $1+y_1 + y_1 y_2$ in
\cref{eq:toyeg-modspecfun}. 


The result that
defines the building blocks of those functions are referred as rearrangement lemma:
cf. \cite[Lemma 6.2]{MR2907006} and \cite[Lemma 6.2]{MR3194491}.
A rigours   prove is given
later by Lesch \cite[Corollary 3.5]{leschdivideddifference} with other deep
observations that motivate many of the results in the paper and the sequel
\cite{Liu:2018ab}. More examples of the spectral functions are provided in the
computation of the $V_4$-coefficient on $\mathbb{T}^2_\theta$
\cite{2016arXiv161109815C} and in the spectral geometry of toric noncommutative
manifolds \cite{Liu:2015aa,LIU2017138}. 

The paper is divided into two parts:  \S\ref{sec:funcal-reamt-lemma} and
\S\ref{sec:diff-and-recursive-relations} consist of new upgrades to the
rearrangement lemma while \S\ref{sec:symbloic-computation} and 
\S\ref{sec:FunRel} focus on the applications onto a simplified model of modular
geometry on general noncommutative $m$-tori. 


In \S\ref{sec:funcal-reamt-lemma}, we complete the toolkit of the modular
geometry by providing a different proof for the
rearrangement lemma that works for arbitrary dimension $m$ and heat coefficient
index $j$ (as in \cref{eq:small-time-heat-asym-functional}). 
Moreover, the result \cref{thm:rarr-lem} lands the functions into a family of
hypergeometric functions known as Lauricella functions $F_D^{(n)}$ of type $D$.  
The hypergeometric features,
for non-experts of special functions, 
are displayed as structures that are interested from algebraic geometry point
of view.  Namely, the integral representations are built upon the
boundary hypersurface equations of 
the standard simplexes $\triangle^n$, see 
\cref{eq:n-simplex-para-defn,eq:hypersurfaces-of-n-simplex,eq:omega-alpha-u-defn,eq:B_n(s)-defn}. 

In \S\ref{sec:diff-and-recursive-relations}, we investigate the differential
and recursive relations which are two sides of a coin regarding to the
hypergeometric features, starting with the one and two variable cases in
\S\ref{subsec:gauss-appell-hgfuns}. Combining with Lesch's observation on divided
difference, we obtain a reduction formula (cf. \cref{thm:red-to-GF21})
which unifies and extends the discussion in
\cite[\S5.1,5.2]{leschdivideddifference} and \cite[\S8.2]{2016arXiv161109815C}.
Roughly speaking, the general $H_{\alpha_0, \ldots, \alpha_n}$ consists of 
partial derivatives of $H_{\alpha_0,1,\ldots,1}$  while $H_{\alpha_0,1,\ldots,1}$
are  the $n$-th divided differences of $z^nH_{\alpha_0,1}(z)$. 
We shall provide more functional relations in the second part 
\cite{Liu:2018ab}to make the family 
$H_{\bar \alpha}(\bar z;m;j)$ into a nice basis
for the study of the modular geometry. 

In \S\ref{sec:symbloic-computation}, we test the upgraded
rearrangement lemma on noncommutative tori $C^\infty(\T^m_\theta)$
of arbitrary dimension. Connes's pseudo-differential calculus extends to higher
dimensions in a straightforward manner and agrees with the deformed Widom's
calculus on toric noncommutative manifolds (cf. \cite{Liu:2015aa,LIU2017138})
with respect to a flat connection.
In fact, we take the later point of view 
to take advantage  of the tensor notation.
Another simplification comes from 
replacing the operator $\Delta_\varphi$ by $\Delta_k =k \Delta$ whose symbol
contains only the leading term.
As a result, 
the full expansion of the $b_2$-term in \cref{eq:b2cal-b_2inr} consists only
$6$ terms\footnote{
In \cite{MR3194491,MR3148618}, the computation was carried out in computer
algebra systems. 
 By  carefully looking at the output \cite[\S6]{MR3194491}
\cite[\S4]{MR3148618}, one might notice that, many of the terms,  are indeed
expansions of tensor contractions.  }.
The main result of the section is the full expression of the modular curvature
$R_{\Delta_k}$ in \cref{thm:psecal-R_Delta_k}. In particular,  the
spectral functions $K_{\Delta_k}$ and $H_{\Delta_k}$ are written in terms of 
the hypergeometric family $H_\alpha(\bar z;m)$.

In the last section, we carefully examine some functional relations 
coming from the modular geometry.
The one behind the Gauss-Bonnet theorem on $\T^2_\theta$ only involves one
variable functions.  
We give a combinatorial verification based on the relations among
the Gauss hypergeometric functions 
discussed in \S\ref{sec:diff-and-recursive-relations}. 
Thanks to the reduction formula in \cref{thm:red-to-GF21},
we are able to verify the
Connes-Moscovici type functional relations for all $m$, 
which strongly validates all the algorithms and technical results 
involved in the computation of the algebraic expressions in 
\cref{eq:checkrels-explict-KDelta,eq:checkrels-explict-HDelta}.
Indeed, numerous algebraic mistakes had been  corrected in the process of
achieving the cancellation.
The result (\cref{thm:-funrel-all-m}) is not only independent of, but also
alludes to the main topic of the sequel paper \cite{Liu:2018ab}:
the variational interpretation of the functional
relations. Moreover, the symbolic verification is stronger in the sense that 
it still holds when  $m$  take real values, 
while in the variational interpretation, $m$ stands for the dimension of the
noncommutative torus which is an integer. 


\subsection*{Acknowledgement}
The author would like to thank Alain Connes, Henri Moscovici and Matthias Lesch
for their comments and inspiring conversations, furthermore is greatly indebted
for their consistent support and encouragement on this project.
The author would also like to thank Farzad Fathizadeh for comments and discussions.
Finally, valuable suggestions from the referee are greatly appreciated.

 The work was carried out during his  postdoctoral fellowship at
MPIM, Bonn. The author would like to thank MPI for its marvelous working
environment.
The author would also like to thank IHES for kind support and excellent working
environment during his visit in November and December, 2017.


\section{Functional Calculus and Rearrangement Lemma}
 \label{sec:funcal-reamt-lemma}


 \subsection{Smooth Functional Calculus on $\mathcal A^{\otimes (n+1)}$}
 \label{subsec:smooth-funcal-A-tensor-n-times}

Our exploration of the rearrangement process is built  upon the abstract setting in 
\cite{leschdivideddifference}. 
 Let $A$  be a $C^*$-algebra. Motivated by \cite[Lemma 6.2]{MR3194491}, Lesch
 considered the following contraction map $\cdot: A^{\otimes n+1} \times
 A^{\otimes  n}\to  A$,  
 on elementary tensors:
\begin{align}
    (a_0 \otimes  \cdots \otimes a_n) \cdot 
(\rho_1\otimes  \cdots  \otimes  \rho_n)
    = a_0 \rho_1 a_1
    \cdots a_{n-1} \rho_n  a_n .
    \label{eq:intro-defn-contraction}
\end{align}
It induces a map
\begin{align}
    \label{eq:iota-from-contraction-defn}
    \iota: A^{\otimes n+1} \to  L(A^{\otimes n}, A)
\end{align}
that sends
 elements in $A^{\otimes  n+1}$ to linear operators from 
$A^{\otimes n} \to A$.

Similar to left and right multiplications,
for any $a \in  A$,  it yields $(n+1)$ operators in $L(A^{\otimes n}, A)$  
\begin{align}
    \label{eq:defn-partial-a-j}  
    a^{(j)} =
    \iota \brac{
    1\otimes  \dots\otimes a\otimes \dots\otimes 1 
    }, \, \, \, \, 
    0\le  j \le n,
\end{align}
where $a$ appears in the $j$-th factor.
In other words, $a^{(j)}$ is the multiplication at the $j$-th slot  when acting
on a elementary tensor $ (\rho_1, \ldots, \rho_n) \in
A^{\otimes  n}$. 
Let $h = h^* \in  A$ be a self-adjoint element with the exponential $k = e^h$.
Denote by $\mathbf x_h, \mathbf y_h \in L(A,A)$ the commutator and conjugation
operator respectively:
\begin{align}
    \label{eq:xh-yh-defn}
    \mathbf x_h(a) = [a, h], \, \, 
    \mathbf y_h (a) = e^{\mathbf x_h} = k^{-1} (a) k, \, \,  \forall a \in  A.
\end{align}
For lifted $\mathbf x_h^{(j)}, \mathbf y_h^{(j)} \in  L(A^{\otimes  n}, A)$, 
the superscript $j=1,\ldots,n$ indicates that the operator 
is applied only onto the $j$-th factor of elementary  tensors like $\rho_1
\otimes  \cdots \otimes \rho_n$. 
More precisely, we have
\begin{align}
    \label{eq:xj-y-to-kj-hj}
    \mathbf x_h^{(j)} = -h^{(j-1)} + h^{(j)} , 
    \, \, \, \, 
    \mathbf y_h^{(j)} = ( k^{(j-1)} )^{-1} k^{(j)}, \, \, 
    j=1,\ldots,n.
\end{align}
In other words,
\begin{align}
    \label{eq:kj-hj-to-xj-and-yj}
    h^{(j)} = h^{(0)} + \mathbf x^{(1)}_h + \cdots + \mathbf x_h^{(j)},
    \, \, \, \, 
    k^{(j)} = k^{(0)}  \mathbf y^{(1)}_h  \cdots  \mathbf y_h^{(j)}
    ,
\end{align}
where $j=1,..,n$. With the new notations in hand, 
the shortest  example of rearrangement proceeds like this 
$\rho k = k^{(1)}(\rho) = k \mathbf y_h(\rho)$. 

We recall the continuous functional calculus for  normal elements
given by the inverse of the Gelfand map.  
For a fixed normal element $h \in  A$, that is $[h,h^*]=0$, 
there is a $*$-isomorphism, called the spectral measure,
\begin{align}
    \label{eq:spetral-measure-Psi-single}
  \Psi: C(X_h) \to C^*(1,h),  
\end{align}
where the domain is the $C^*$-algebra of continuous functions on 
$X_h = \op{Spec}\,  h \subset
\mathbb{C}$, the spectrum of $h$ and the range is the   $C^*$-subalgebra of $A$
generated by $h$  and $1$.   
For any $f \in  C(X_h)$, the functional calculus, that is the evaluation of $f$
at $h$ is defined to be 
\begin{align}
  f(h) \defeq \Psi(f) \in C^*(1,h) \subseteq A . 
  \label{eq:f(h)=Psi(f)}
\end{align}

We also need to consider different topological tensor products  
before proceeding to the functional calculus for multivariable functions: 
the maximal $C^*$-tensor product  
$\otimes_\pi$ and the projective tensor product  $\otimes_\gamma$, which are norm
completions of the  algebraic one $A^{\otimes  n}$. 
For $A^{\otimes_\gamma n}$, we use  
\begin{align}
    \label{eq:gamma-norm}
    \norm{a}_\gamma = \inf \sum_s \norm{\alpha_0^{(s)}} \cdots
    \norm{\alpha_n^{(s)}}
\end{align}
where the infimum is taken over all possible representations of $a$ as finite
sums of elementary tensors $a = \sum_s \alpha_0^{(s)} \otimes  \cdots  \otimes
\alpha_n^{(s)}$. 
For $A^{\otimes_\pi n}$, the norm is the sup among  all $*$-representations of
$A^{\otimes  n}$:
\begin{align}
    \norm{a}_\pi = \sup \norm{ \rho(a)}.
    \label{eq:pi-norm}
\end{align}

The two norms in \cref{eq:pi-norm,eq:gamma-norm} satisfy $\norm{\cdot }_\pi \le
\norm{\cdot }_\gamma$. Hence the identity map on the algebraic tensor product
extends to a continuous surjective $*$-homomorphism 
\begin{align}
  p^\gamma_\pi: A^{\otimes_\gamma n+1} \to  A^{\otimes_\pi n+1} . 
  \label{eq:p-gamma-to-pi}
\end{align}
Accordingly,
$A^{\otimes_\pi n+1}$  is a quotient of $A^{\otimes_\gamma n+1}$.
The two   topological tensor products come with  somehow mutually
exclusive features that
are essential for constructing functional calculi.
The maximal $C^*$-tensor product $\otimes_\pi$ has the advantage that
the completion $A^{\otimes_\pi  n}$ is still a $C^*$-algebra. 
For the  $C^*$-algebra  $C(X_h)$ in \cref{eq:f(h)=Psi(f)}, the maximal
$(n+1)$-fold tensor product is isomorphic to the algebra of continuous
functions on the $(n+1)$-fold Cartesian
product of $X_h$: $C(X_h)^{\otimes_\pi  n+1} \cong C(X_h^{n+1})$.
Similar to \cref{eq:f(h)=Psi(f)},
for $h^{(0)}, \ldots, h^{(n)} \in  A^{\otimes n+1}$\footnote{
    Strictly speaking, here $h^{(j)}$ means the elementary tensor 
    $1\otimes \cdots  h \cdots \otimes 1$ before applying the $\iota$ map, 
    cf. \cref{eq:defn-partial-a-j}. 
}, 
we have the joint spectral measure (a $*$-isomorphism) 
\begin{align}
     \Psi_\pi: C(X_h^{n+1}) \to  C^*(1, h^{(0)}, \ldots, h^{(n)})
\subset A^{\otimes_\pi n+1} ,
     \label{eq:Psi-pi}
 \end{align}
where the $C^*$-algebra generated by  $1,h^{(0)}, \ldots, h^{(n)}$ is sitting
inside the maximal $C^*$-tensor product.  
The drawback of the maximal $C^*$-tensor product is that
whether the multiplication map $\mu_n: A^{\otimes  n} \to  A$ can be
continuously  extended to $A^{\otimes_\pi  n} \to  A$ is still inconclusive  
\footnote{
The author's knowledge on this issue is limited to the discussion in 
\cite[\S3]{leschdivideddifference}
}. 
thus we don't know whether the map $\iota$ in \cref{eq:iota-from-contraction-defn}
extends continuously to 
$A^{\otimes_\pi n+1} \to L_{\mathrm{cont}} (A^{\otimes_\pi n}, A)$. 
To make use of such features, we turn to 
the projective one $A^{\otimes_\gamma n+1}$  that behaves well when
extending the multiplication map $\mu_{n+1}$, in particular, $\iota$ in 
\cref{eq:iota-from-contraction-defn} admits a continuous extension
\begin{align}
    \label{eq:iota-cont-version}
    \iota: A^{\otimes_\gamma n+1} \to  L_{\mathrm{cont}} (A^{\otimes_\pi n}, A)
    .
\end{align}
Nevertheless $A^{\otimes_\gamma n}$  is only a Banach $*$-algebra and almost
never a $C^*$-algebra.

The two features coexist regarding to
smooth functional calculus, thanks to the nuclearity of the
Fr\'echet topology of $ C^\infty(U) $, where   
$U \supset X_h $ is an open set in $\R$  containing the spectrum of 
$h =h^*\in  A$.
The nuclearity
 states that the projective $\otimes_\gamma$ and the injective
$\otimes_\epsilon$ tensor products agree and they both yield:
\begin{align*}
    C^\infty(U)^{\otimes_\gamma n+1} \cong C^{\infty}(U^{n+1}) \cong 
C^\infty(U)^{\otimes_\epsilon n+1}.
\end{align*}
The injective feature of $C^{\infty}(U^{n+1})$
allows us to approximate multivariable functions $f(x_0, \ldots, x_n)$ by those of
the form of separating variables. 
 In more detail,
the algebraic map
$C^\infty(U)^{\otimes n+1} \to  C^\infty(U^{n+1})$, where 
$f_0 \otimes
\cdots  \otimes f_n \to  f$ with $f(x_0, \ldots, x_n) = f(x_0) \cdots  f(x_n) \in
C^\infty(U^{n+1})$ extends by continuity to the isomorphism of 
$C^\infty(U)^{\otimes_\epsilon n+1} \to  C^\infty(U^{n+1})$. 
By performing the projective completion onto the algebraic map
\begin{align}
    \label{eq:Psi-otimes-n+1}
    \Psi^{\otimes n+1}: C^\infty(U)^{\otimes n+1} \to A^{\otimes  n+1}:
    f_0\otimes \cdots \otimes  f_n \to  f_0(h) \otimes \cdots  \otimes f_n(h),
\end{align}
we obtain a continuous map landed in $A^{\otimes_\gamma  n+1}$: 
\begin{align}
    \label{eq:Psi-gamma-defn}
    \Psi_\gamma: C^\infty(U^{n+1}) \to A^{\otimes_\gamma  n+1}.
\end{align}
\begin{defn}[Smooth functional calculus]
    \label{defn:f-gamma}
    For any $f \in  C^\infty(U^{n+1})$ and
    self-adjoint $h \in A$ , 
    the smooth functional calculus denoted by 
    $f_\gamma(h^{(0)}, \ldots, h^{(n)})$  takes value in operators from 
$A^{\otimes_\gamma n}$ to $A$:
    \begin{align}
        \label{eq:f-gamma=Psi-gamma(f)}
        f_\gamma(h^{(0)}, \ldots, h^{(n)}) \defeq  
        (\iota\circ \Psi_\gamma)(f) 
        \in L_{\mathrm{cont}} (A^{\otimes_\gamma n},A)  
        ,
    \end{align}
    where $\iota$ and $\Psi_\gamma$ are given in 
    \cref{eq:iota-cont-version,eq:Psi-gamma-defn} respectively. 
    Later, we may use the following integral notation for the spectral measure
    \begin{align}
        \label{eq:dE-gamma-x}
        \Psi_\gamma(f) \defeq \int_{U^{n+1}} f(\bar x) dE^{\gamma}(\bar x) ,
    \end{align}
    where $dE^\gamma(\bar x)$ is the induced distribution (valued in the Banach space
    $A^{\otimes_\gamma n+1}$) on $U^{n+1}$.  
\end{defn}

To summarize the long discussion of the two functional calculi $\Psi_\pi$ and
$\Psi_\gamma$ (cf.  \cref{eq:Psi-pi,eq:Psi-gamma-defn}),
we mention two properties starting with  the compatibility.
\begin{prop}[Lesch]
Let $\op{Res}^U_{X_h}: C^\infty(U^{n+1}) \to  C^\infty(X_h^{n+1})$ be the
restriction map $f \to  f|_{X_h^{n+1}}$. We have 
$\forall  f \in
C^\infty(U^{n+1})$, 
\begin{align}
    \label{eq:comaptibility-smooth-cont}
    \Psi_\pi \circ \op{Res}^U_{X_h} (f) = p^\gamma_\pi \circ \Psi_\gamma(f)
    \in  A^{\otimes_\pi n+1},
\end{align}
where $p^\gamma_\pi$ is the natural projection in \cref{eq:p-gamma-to-pi}. 
\end{prop}
\begin{proof}
    See \cite[\S3]{leschdivideddifference}.  
\end{proof}

The other is a Fubini type result. When the spectral function 
is given as an integral:
\begin{align}
    \label{eq:F-bar-x}
   \bar x \in  U^{n+1}\to 
    F(\bar x) = 
    \int_{\mathcal B} f(p,\bar x)d\nu_p , 
\end{align}
where $(\mathcal B, \nu)$ is a Borel space. We expect that the evaluations 
 $\bar x =(h^{(0)}, \ldots, h^{(n)})$ 
 before or after the $\int_{\mathcal B}$ yield the same result:
\begin{align*}
    \Psi_\gamma(F)=  F_\gamma(h^{(0)}, \ldots, h^{(n)}) = 
    \int_{\mathcal B}
    f_\gamma(p,\cdot )
    d\nu_p = 
    \int_{\mathcal B} f(p,h^{(0)}, \ldots, h^{(n)})d\nu_p , 
\end{align*}
where the right hand side is  a Bochner integral valued in
$A^{\otimes_\gamma n+1}$. 
After replacing the left hand side with the integral notation 
\cref{eq:dE-gamma-x}, we see that it is an issue of  switching  order of
integrations:
\begin{align}
    \label{eq:switching-dE-dp}
    \int_{U^{n+1}} \int_{\mathcal B} f(p,\bar x) d\nu_p dE^{\gamma}(x) 
    =
    \int_{\mathcal B} 
    \int_{U^{n+1}}
    f(p,\bar x) 
    dE^{\gamma}(x) 
    d\nu_p .
\end{align}
The following two conditions on $f$ are sufficient for the validation of
\cref{eq:switching-dE-dp} (cf. \cite[Theorem 3.4]{leschdivideddifference}):
\begin{enumerate}
    \item $f:\mathcal B \times U^{n+1} \to \mathbb{C}$ is continuous in
        $(p,\bar x)$ and smooth in $\bar x$;
    \item For any compact set $K\subset U$, we have the integrability for all
        derivatives (in $\bar x$), namely, for any multiindex $\alpha$, 
        \begin{align}
            \label{eq:integrability-f-p-x}
            \int_{\mathcal B} 
            \sup_{\bar x \in  K}
            \abs{ \partial^{\alpha}_{\bar x}f(p,x)} d\nu_p < \infty .
        \end{align}
\end{enumerate}
For the continuous functional calculus $\Psi_\pi$, we can ask for the corresponding
version of \cref{eq:switching-dE-dp} and only the integrability  for $f$ is
required. In fact, integrability for all derivatives as in
\cref{eq:integrability-f-p-x} is used to guarantee the converges of the
integral  $\int_{\mathcal B} f(p, \cdot ) d\nu_p$ in the Fr\'echet topology of
$C^{n+1}(U)$, that is the smoothness of $F(\bar x)$ in \cref{eq:F-bar-x}.   

For readers only concerning the applications of the
lemma and prefer  avoiding completions of the algebraic tensor product, there is 
a much more elementary way to define the rearrangement operators using Schwartz
functions.
\begin{defn}
    \label{defn:schwartz-funcal}
For a fixed self-adjoint $h \in  A$ with $k =e^h$, we have a  map from the space of
Schwartz functions of $n$-variables  to the space of linear maps from
$A^{\otimes  n}\to  A$: 
\begin{align*}
    \Psi_{\mathscr S}:\mathscr S(\R^n) \to  L(A^{\otimes  n}, A):
    f \to  \Psi_{\mathscr S}(f) \defeq
    f_{\mathscr S}(\mathbf x_h^{(1)}, \ldots, \mathbf x_h^{(n)}),
\end{align*}
where on elementary tensors, the operator is evaluated as below:
\begin{align}
    \label{eq:f(bfx_1-bfx_n)-defn}
    &\, \, 
   f_{\mathscr S}(\mathbf x_h^{(1)}, \ldots, \mathbf x_h^{(n)})
    \brac{ \rho_1 \otimes  \cdots  \otimes  \rho_n} 
    \\
=&\, \,  
    \int_{\R^{n}} \hat f(\xi) 
    \brac{
        e^{i\xi_1 \mathbf x_h^{(1)}} \cdots e^{i \xi_n \mathbf x_h^{(n)}}
    }
    \brac{ \rho_1 \otimes  \cdots  \otimes  \rho_n} 
  =
    \int_{\R^{n}} \hat f(\xi) 
    \mathbf y_{h}^{i\xi_1}(\rho_1) \cdots 
    \mathbf y_{h}^{i\xi_n}(\rho_n) d\xi,
    \nonumber
\end{align}
where $f$  and $\hat{ f}$ are related  in the following way:
\begin{align}
    \label{eq:f=int-hat(f)}
    f( x) = \int_{\R^{n}} \hat f(\xi) e^{i \xi \cdot x} d\xi.
\end{align}
By a change of variable $y=e^x$, we can define functional calculus for the
conjugation operators
$f_{\mathscr S}(\mathbf y_h^{(1)},\ldots, \mathbf y^{(n)}_h) \in  L(A^{\otimes
n},A)$ by
imposing
\begin{align*}
  f_{\mathscr S}(\mathbf y_h^{(1)},\ldots, \mathbf y^{(n)}_h) \defeq 
  (f_{\log})_{\mathscr S}
  (\mathbf x_h^{(1)}, \ldots, \mathbf x_h^{(n)})
\end{align*}
whenever $f_{\log}$ belong to $\mathscr S(\R^n)$, where
$ f_{\log}(x_1, \ldots, x_n) = f(e^{x_1}, \ldots, e^{x_n})$.
\end{defn}
\begin{rem}
    \hspace{2cm} \\
    \begin{enumerate}
        \item 
To see \cref{defn:schwartz-funcal} (or \cref{eq:f=int-hat(f)})
 is a special case of the smooth functional calculus, 
we have, for $h \in  A$ self-adjoint and $h^{(0)}, \ldots, h^{(n)} \in
A^{\otimes n+1}$  
\begin{align*}
    \Psi_\gamma \brac{
        \int_{\R^{n+1}} \hat{f}(\xi) e^{i x \cdot  \xi} d\xi
    } 
    & =
    \int_{\R^{n+1}} \hat{f}(\xi)
    \Psi_\gamma \brac{ 
        e^{i x_1 \cdot  \xi_1}  \cdots 
        e^{i x_n \cdot  \xi_n}  
    } d\xi \\
    &=  
    \int_{\R^{n+1}} \hat{f}(\xi)
    (k^{(1)})^{i\xi_1} \cdots 
    (k^{(n)})^{i\xi_n} d\xi
\end{align*}
where $k = e^h$ as before. Notice that the Fubini property
\cref{eq:switching-dE-dp} is applied and the integrability
\cref{eq:integrability-f-p-x} follows from the Schwartzness of $f$. 

\item Since the spectrum of $\mathbf x_h$ is always bounded, the Schwartzness
    is guaranteed. In fact,
    for any smooth
    $f(x_1, \ldots, x_n)$, we can replace $f$ by $\psi f$ where $\psi$ is
    a suitable   cutoff function $\psi$ so that $\psi
    f$ has compact support which contains the joint spectrum. Due to the
    compatibility \cref{eq:comaptibility-smooth-cont},
    the result is
    independent of $\psi$ as long as $\op{Res}^U_{X_h}(\psi f)$ is the same.      
    \end{enumerate}
\end{rem}

\subsection{Rearrangement Lemma}
In \S\ref{subsec:application-of-rearr-lem}, one encounters the following 
family of integrals:
\begin{align*}
    \int_0^\infty G_\alpha(r \bar s) r^d dr
\end{align*}
where $\alpha = (\alpha_0, \ldots, \alpha_n) \in  \Z^{n+1}_+$ is a multiindex,
$\bar s= (s_0, \ldots, s_n) \in (0,\infty)^{n+1}$, $d>0$ depends on the
dimension and the index of the aimed heat coefficient, 
and  
  \begin{align}
      \label{eq:Galpha-as-a-contour-int}
  G_{\alpha_0,\dots,\alpha_n} (\bar s) 
  = 
 (2\pi i)^{-1}
\int_C
e^{-\lambda}
\brac{
    \prod_{j=0}^{n}   (s_j - \lambda )^{-\alpha_j}
}
d\lambda,
  \end{align}
  where the contour $C$ is the same one used in the holomorphic functional
  calculus devotion of the heat operator.
  The argument  given in \cite[\S6]{MR3194491} relies on exchanging the order
  of integration  $d\lambda dr \to  dr d\lambda$ which is special to
  noncommutative two tori, see the remark at the end of
  \S\ref{subsec:application-of-rearr-lem} for more details.

  The first contribution of the paper is to show that the integrals are indeed
  some hypergeometric functions.
  For non-experts of special functions, the hypergeometric features are
  displayed as the appearance of 
equations of the boundary hypersurfaces:
\begin{align}
    \label{eq:hypersurfaces-of-n-simplex}
  1 - \sum_1^n u_j =0, u_1= 0, \ldots, u_n=0  
\end{align}
of the standard $n$-simplex $\triangle^n \subset \R^n$: 
\begin{align}
    \label{eq:n-simplex-para-defn}
    \triangle^n = \{
        u=(u_1,\ldots, u_n) \in  \R^n: \, \sum_1^n u_j \le 1, 
        u_1\ge  0, \ldots, u_n \ge 0
    \} .
\end{align}
For $\alpha = (\alpha_0, \ldots, \alpha_n) \in  \Z_+^{n+1}$ and
$s =(s_0, \ldots, s_n) \in (0,\infty)^{n+1}$, 
we introduce the following weighted product $\omega_\alpha(u)$ 
and the weighted sum $B_n(\bar s,u)$ of the hypersurface equations
listed  \cref{eq:hypersurfaces-of-n-simplex}:
\begin{align}
    \label{eq:omega-alpha-u-defn}
    \omega_\alpha(u) & = 
    \frac{1}{\Gamma(\alpha_0) \cdots  \Gamma(\alpha_n)}
    (1 - \sum_1^n u_j)^{\alpha_0-1} \prod_{s=1}^{n}
    u_s^{\alpha_s-1} ,
    \\
    B_n(\bar s,u) &= s_0(1-\sum_{l=1}^n u_l) + \sum_{l=1}^n s_l u_l
    \label{eq:B_n(s)-defn}
    .
\end{align}

\begin{lem}
\label{lem:hgeofun-contourint-nvar-case}
Keep the notations. 
 Let us take contour $C$ in \cref{eq:Galpha-as-a-contour-int} to be the imaginary
 axis oriented from $-i\infty$ to $i\infty$:    
  \begin{align}
       \label{eq:Galpha-from-contour-int}
  G_{\alpha_0,\dots,\alpha_n} (\bar s) 
  = 
 (2\pi)^{-1}
\int_{-\infty}^\infty e^{-i x}
\brac{
    \prod_{j=0}^{n}   (s_j - ix)^{-\alpha_j}
}
dx.
  \end{align}
  It admits another integral representation:
\begin{align}
    \label{eq:hggeo-Ga0toan}
  G_{\alpha_0,\dots,\alpha_n} (\bar s) 
  = 
\int_{\triangle^n} \omega_\alpha(u)
e^{-B_n(\bar s,u)}
    du ,
\end{align}
where $\omega_\alpha(u)$ and $B_n(\bar s,u)$ are
 defined in \cref{eq:omega-alpha-u-defn,eq:B_n(s)-defn} respectively.
\end{lem}
\begin{rem}
 When $n=1$,  the  right hand side of  \cref{eq:hggeo-Ga0toan} is a confluent
    hypergeometric function:
    \begin{align}
        (2\pi)^{-1} \int_{-\infty}^{\infty} e^{-ix} 
        (A - ix)^{-a} ( B - ix)^{-b} dx  =  \frac{e^{-A}}{ \Gamma(a+b)} 
        \pFq11(a;a+b;B-A),
        \label{eq:hgeofun-F11-heat-contourint}
    \end{align}
    where the function $ \pFq11(a;b;B-A)$ has the following integral
    representation:
    \begin{align}
        \pFq11(a;b;z)= \frac{
            \Gamma(a-b) \Gamma(b)
        }{ \Gamma(a)}\int_0^1 t^{a-1} (1-t)^{b-a-1} e^{zt} dt,
        \label{eq:hgeofun-F11-int-rep-defn}
    \end{align}
     which appeared   in the
     heat kernel related work \cite{MR1156063} and \cite{MR1843855}. This was the
     motivation in the early stage that brought the author's attention to 
     hypergeometric functions. 
 \end{rem}
\begin{proof}
    For the case $n=0$, a residue computation  shows that
    $\forall  a \in  \Z_{+}, s>0$,  
    \begin{align*}
        (2\pi)^{-1}
        \int_{-\infty}^\infty e^{-ix} (s-ix)^{-a} dx =
        \frac{e^{-s}}{\Gamma(a)} \mathsf H(s) ,
    \end{align*}
where $\mathsf H(\eta)$ is the  Heaviside step function which takes value $1$ on
$(0,\infty)$ and $0$ on $(-\infty,0)$.    
Let $\mathcal F_x$ denote the Fourier transform in $x$,  
    by a change of variable ($x\to  x \eta$) argument, we have for $s>0$:
\begin{align*}
    \mathcal F_x \brac{ (s-ix)^{-a} } (\eta) = 
    (2\pi)^{-1} 
    \int_{-\infty}^\infty e^{-ix \eta} (s-ix)^{-a} dx =
        \eta^{a-1} e^{-s\eta} \mathsf H(\eta) .
\end{align*}
The result follows from the well-known fact that Fourier transform turns
pointwise multiplication into convolution ($*$)\footnote{
The author would like to thank the referee for pointing out 
    the convolution observation. 
}.  In detail, 
let $f_j(\eta)  =\eta^{\alpha_j-1} e^{-s_j\eta} \mathsf H(\eta)$ 
be the Fourier transforms of $(s_j - ix)^{-\alpha_j}$ 
where $j=0,\ldots,n$, then  
\begin{align*}
    &\, \, 
    (2\pi)^{-1}
\int_{-\infty}^\infty e^{-i x}
(s_0 -ix)^{-\alpha_0} \cdots (s_n -ix)^{-\alpha_n} dx =
\mathcal F_x \brac{
\prod_{j=1}^{n} (s_j-ix)^{-\alpha_j}  
} (1)
\\
    =&\, \, (f_0 * \cdots  * f_n) (1)
    =
    \int_{\R^{n}} f_0(1-\sum_1^n u_j) 
    \brac{\prod_{1}^{n} f(u_j)} du 
    \\
    =&\, \, \int_{\R^{n}}
    \omega_\alpha(u) e^{
    -s_0(1-\sum_{l=1}^n u_l) - \sum_{l=1}^n s_l u_l
    } 
    \sbrac{
    \mathsf H \brac{ 1-\sum_1^n u_j} 
    \brac{\prod_{1}^{n} \mathsf H(u_j)} 
    }
    du
    \\
    = &\, \, 
\int_{\triangle^n} \omega_\alpha(u)
    \cdot \exp\brac{
        -s_0(1-\sum_{l=1}^n u_l) - \sum_{l=1}^n s_l u_l
    } du .
\end{align*}
\end{proof}


    We fix a self-adjoint element $h = h^* \in  A$ in a $C^*$-algebra $A$ and
    denote the exponential by $k = e^h$. Let $h^{(l)}$,  $k^{(l)}$, $\mathbf
    x^{(l)}$ and $\mathbf y^{(l)}$, $0\le l\le n$,  be the operators 
introduced in 
$\S\ref{subsec:smooth-funcal-A-tensor-n-times}$ with additional equivalent
family:
$\mathbf{\vec z} = (\mathbf z^{(1)}, \ldots, \mathbf z^{(n)})$ where 
\begin{align}
    \label{eq:mathbf-z^(n)-defn}
    \mathbf z^{(l)} = 1- \mathbf y^{(1)} \cdots  \mathbf y^{(l)}, \, \, \, \, 
    l=1,\ldots,n.
\end{align}
\begin{thm}[Rearrangement Lemma]
    \label{thm:rarr-lem}
For a multiindex $\alpha = (\alpha_0, \ldots, \alpha_n) \in  \Z_+^{n+1}$ and
$d>0$ , we have 
the following equality for the two Bochner integrals valued in the Banach
$*$-algebra $A^{\otimes_\gamma n+1}$: 
\begin{align}
    \label{eq:contourint-to-tild-H-alpha}
    (2\pi)^{-1}
    \int_0^\infty 
    \int_\infty^\infty e^{-ix}
    \prod_{l=0}^{n}  
    \brac{
        k^{(l)} r - ix
    }^{-\alpha_l}
    r^{d}
    dx dr = k^{-(d+1)}
  (\tilde  H_\alpha)_\gamma (\mathbf{\bar z};d),
\end{align}
where $k^{-(d+1)} \defeq (k^{(0)})^{-(d+1)}$ is the left multiplication and the
spectral function is given by a hypergeometric integral:
\begin{align}
    \label{eq:tild-H-alpha-z-d-defn}
    \tilde  H_\alpha( \bar z;d) = \Gamma(d)
    \int_{\triangle^n} \omega_\alpha(u)
    \brac{
        1-  \bar z \cdot  u
    }^{-(d+1)}
    du,  \, \, \, \, 
    \bar z \in  \R^{n} ,
\end{align}
where $\omega_\alpha(u)$ is the weighted product defined in
\cref{eq:omega-alpha-u-defn}. 
\end{thm}
\begin{proof}
According to Lemma \ref{lem:hgeofun-contourint-nvar-case} to handle the integral  
 $\int^\infty_{-\infty}(\cdot )dx$: 
    \begin{align*}
        &\, \, 
        (2\pi)^{-1} \int_0^\infty \int_\infty^\infty e^{-ix}
    \prod_{l=0}^{n}  
    \brac{
        k^{(l)} r - ix
    }^{-\alpha_l}
    r^{d}
    dx dr 
    = 
    \int_0^\infty 
    G_{\alpha_0, \ldots, \alpha_n}(rk^{(0)}, \ldots, rk^{(n)}) 
    r^d dr
    \\ 
    = &\, \, 
   \int_0^\infty \int_{\triangle^n}
   \omega_\alpha(u) e^{- B_n(r\vec k,u)} r^d dr  
=   \int_{\triangle^n} \omega_\alpha(u)
\sbrac{
\int_0^\infty r^d  e^{-r B_n(\vec k,u)} dr
} du ,
    \end{align*}
    where $\vec k = (k^{(0)}, \ldots, k^{(n)})$. 
Let $U \subset (0,\infty)$ be an open set containing the spectrum of $k$, 
 for $\bar s \in  U^{n+1}$ and $u \in  \triangle^n$, the weighted sum defined
 in \cref{eq:B_n(s)-defn} is positive:
 \begin{align*}
     B_n(\bar s,u) = s_0 (1-\sum_1^n u_l) + \sum_1^n s_l u_l >0 .
 \end{align*}
 In particular, we see that $B_n(\vec k,u) \in  A^{\otimes_\gamma n+1}  $ has
 positive spectrum.
A change of variable argument gives:
\begin{align}
    \label{eq:F-B_n-s-u}
    F(\bar s) = \int_0^\infty r^d e^{-rB_n(s,u)} dr = (B_n(\bar s, u))^{-(d+1)} 
    \int_0^\infty r^d e^{-r} dr =  (B_n(\bar s, u))^{-(d+1)} \Gamma(d),
\end{align}
so that 
\begin{align*}
    \int_0^\infty r^d  e^{-r B_n(\vec k,u)} dr = F( k^{(0)}, \ldots, k^{(n)} )
    = \brac{ B_n(\vec k,u)}^{-(d+1)} \Gamma(d),
\end{align*}

The new variable $\mathbf{\bar z} =( \mathbf z^{(1)}, \ldots, \mathbf z^{(n)} )$ 
arises from 
replacing $\{k^{(l)}\}_{l=1}^n$ by the conjugations
$\{\mathbf y^{(l)} \}_{l=1}^n$ according to \cref{eq:kj-hj-to-xj-and-yj}:
\begin{align*}
    B_n(\vec k,u) & = k^{(0)} (1-\sum_1^n u_l) + \sum_1^n k^{(l)} u_l 
    = k^{(0)} \brac{
        1- \sum_1^n 
        \brac{
        1- ( k^{(0)})^{-1} k^{(l)} 
        } u_l
    }
    \\  
    &= 
    k^{(0)} \brac{
        1- \sum_1^n 
        \brac{
            1- \mathbf y^{(1)} \cdots \mathbf y^{(l)}
        } u_l
    }  
    =
k^{(0)} \brac{
        1- \sum_1^n 
        \mathbf z^{(l)} u_l
    }  .
\end{align*}
Finally, 
\begin{align*}
 &\, \, 
 (2\pi)^{-1} \int_0^\infty \int_\infty^\infty e^{-ix}
    \prod_{l=0}^{n}  
    \brac{
        k^{(l)} r - ix
    }^{-\alpha_l}
    r^{d}
    dx dr 
\\
    = &\, \, 
    \Gamma(d)  k^{-(d+1)} \int_{\triangle^n} \omega_\alpha(u)
    \brac{
        1- \mathbf{\bar z} \cdot  u
    }^{-(d+1)}
    du .
\end{align*}

We point out that, in the calculation above,
the Fubini type result \cref{eq:switching-dE-dp} has been quoted  three times:
with respect to the two integral representation
of $G_\alpha(\bar s)$ given in
\cref{eq:Galpha-from-contour-int,eq:hggeo-Ga0toan}, and to the function in 
\cref{eq:F-B_n-s-u}.
Let us exam 
 the integrability condition (cf. \cref{eq:integrability-f-p-x}), 
for the function in \cref{eq:Galpha-from-contour-int} as an example, and leave
other two, which are quite straightforward, to the reader. 
Observe that the integral is a contour integral
along the imaginary axis from $-i\infty$ to $i\infty$. In order to gain
exponential decay, we change the contour to the another one inside its
homologous class, say the straight lines from $\infty - i$ to $0$ and then to
$\infty+i$.     
The integral in \cref{eq:hggeo-Ga0toan} reads, with $ix \to  \lambda$:
\begin{align*}
    \abs{
  \partial^\beta_{\bar s}  e^{-\lambda} 
    \prod_{l=0}^{n}  
    \brac{
         s_l - \lambda
    }^{-\alpha_l}
    } < C_{\bar s, \beta} e^{- \Re\lambda},
\end{align*}
here  $\lambda$ is the parameter of  the later contour so that $\Re \lambda \to
\infty$ at the both ends.

\end{proof}

\section{Differential and Recursive Relations among the Hypergeometric Family}
\label{sec:diff-and-recursive-relations}

The straightforward consequence from the hypergeometric feature is the fact
that differential and contiguous relations
are the two sides of a coin among the family. It leads to new paths to compute
the algebraic expressions of the family $H_{\alpha}(\bar z;m;j)$. 
The additional ingredient is the divided difference operation, which,
according to the author's limited knowledge, does not come from  the 
literature of special functions.


\subsection{Lauricella Functions}
In the companion in \S\ref{subsec:application-of-rearr-lem}, the exponent $d$ in 
\cref{eq:tild-H-alpha-z-d-defn} is determined by the multiindex $\alpha$, the
dimension $m$   and the heat coefficient index
$j$ (see \cref{eq:small-time-heat-asym-functional}). 
\begin{defn}
As shown in \cref{prop:R^j_P-general-form},
the exact building blocks of the spectral functions 
is indexed by $\alpha \in  \Z^{n+1}_+$: 
\begin{align}
    &\, \,  
    H_\alpha(\bar z;m;j) =
    \tilde H_\alpha(\bar z; d(\alpha;m;j) ) 
    \nonumber \\
    = &\, \, 
\Gamma(d(\alpha;m;j) +1)
  \int_{\triangle_n}  \omega_\alpha(u)
    \brac{
        1- \sum_{l=1}^n z^{(l)} u_l
    }^{-d(\alpha;m;j) -1}
    du,
    \label{eq:H-alpha-defn}
\end{align}
where $\bar z = (z^{(1)}, \ldots , z^{(n)})$ and  
\begin{align}
    \label{eq:d-alpha-m-j-defn}
    d(\alpha;m;j) =  \abs\alpha +\frac{m - j }{2} -2, \, \, \, \, 
    \text{where} \, \,  \abs \alpha= \sum_{s=0}^{n} \alpha_s
    .
\end{align}
We often drop the argument $j$ with the default value $j=2$ since 
second heat coefficient is the main focus of the paper:
\begin{align}
    \label{eq:H-alpha--m-j=2}
    H_\alpha(\bar z;m) \defeq 
    H_\alpha(\bar z;m;2 )  .
\end{align} 
\end{defn}

The functions  belong to the hypergeometric family known as Lauricella
functions $F^{(n)}_D$ of type $D$
\footnote{
Here $D$ is merely a label.
}, which are multivariable generalizations of the Gauss hypergeometric functions. 
The two variable case was first studied by Appell and 
the three variable case was first introduced
by Lauricella in \cite{Lauricella1893} and later 
fully by Appell and Kamp\'e de F\'eriet \cite{appell1926fonctions}.
They are extensions of the hypergeometric series 
\begin{align}
    \label{eq:hgeofun-LauricellaSeries}
     F^{(n)}_D (a;\alpha_1,\dots,\alpha_n;c;x_1,\dots,x_n) 
    =  \sum_{\beta_1,\dots,\beta_n\ge 0} \frac{
        (a)_{\beta_1+\cdots+\beta_n} (\alpha_1)_{\beta_1}
        (\alpha_n)_{\beta_n} x_1^{\beta_1}\cdots x_n^{\beta_n}
    }{ 
(c)_{\beta_1+\cdots+\beta_n}
     \beta_1!\cdots \beta_n!
 },
\end{align}
where the Pochhammer symbols is defined as:
\begin{align}
    (q)_n = \frac{ \Gamma(q+n)}{ \Gamma(q)}.
\end{align}
What we need in the paper is the following integral representation (cf.
\cite{MR0346178}):
\begin{align}
    \label{eq:hgeofun-LauricellaFuntions}
    F^{(n)}_D (a;\alpha_1,\dots,\alpha_n;c;\bar z) 
    = \Gamma(c)
\int_{\triangle^n}
\omega_{ \tilde \alpha}(u)
        \cdot (1- \bar z \cdot  u )^{-a} du
        ,
       \end{align}
where $ \tilde \alpha_0 = c- \sum_{l=1}^n \alpha_l$ and
$ \tilde \alpha_1 = \alpha_1, \ldots, \tilde\alpha_n = \alpha_n$.   
 In particular, for $\alpha = (\alpha_0, \ldots, \alpha_n)$ and 
$\abs\alpha = \sum_0^n \alpha_l$, we have: 
\begin{align}
    \label{eq:H-alpha-to-FD^(n)}
    H_\alpha(z;m;j) = 
    \frac{\Gamma( \tilde d_{\alpha,m,j}) }{\Gamma\brac{
            \abs\alpha
    }}
    F_D^{(n)}
    ( \tilde d_{\alpha,j,m};\alpha_1, \ldots, \alpha_n; \abs\alpha; \bar z)
    ,
\end{align}
where $\tilde d_{\alpha,j,m}   = \abs\alpha + (m-j) /2 -1$.

 \subsection{Gauss  and Appell  Hypergeometric Functions}
\label{subsec:gauss-appell-hgfuns}

To explore the merits of the hypergeometric features, 
we  begin with the
special cases  that are actually heavily involved in the later
discussion of the modular geometry.



 The one variable family $F_D^{(1)} = \pFq21$ is the
Gauss hypergeometric functions, 
 one of the best-known classes among special functions: 
\begin{align}
 \pFq21(a,b;c; z)
 = \frac{\Gamma(c) }{ \Gamma(b) \Gamma(c-b)}     
 \int_0^1 (1-t)^{c-b-1} t^{b-1} (1-z t)^{-a} dt.
    \label{eq:hgeo-defn-pFq21}   
\end{align}
It is a solution of Euler's hypergeometric differential equation
\begin{align}
    z(1-z) \frac{d^2 w}{dz^2} + \brac{
        c-(a+b+1)z
    } \frac{dw}{dz} - ab w =0 
    \label{eq:hgeofun-HGODE}
\end{align}

Let us abbreviate $F \defeq \pFq21(a,b;c;z)$ for a moment. There are six
associated contiguous functions
obtained by applying $\pm1$ on  the parameters $a,b$ and $c$, such as:
$F(a+)$, $F(b+)$, $F(c-)$, etc.  Gauss showed that $F$ can be written as
a linear combination of any two of its contiguous functions, which leads to
$15$ ($6$ choose $2$) relations.
They can be derived from the differential
relations among the
family 
\begin{align}
    \label{eq:hgfun-diff-relations-GF21-abc+1}
    \frac{d}{dz}  ( \pFq21(a,b;c;z) ) = \frac{a b}{c} \pFq21(a+1,b+1;c+1;z),
\end{align}
also
\begin{align}
    \label{eq:hgfun-diff-relations-GF21-a+1}
    F(a+) &= F + \frac1a z\frac{d}{dz} F, \, \, \, \, 
    F(b+) = F + \frac1b z\frac{d}{dz} F, \\
    F(c-) &= F + \frac1c z\frac{d}{dz} F. 
    \label{eq:hgfun-diff-relations-GF21-c-1}
\end{align}
Combining with the second order ODE \cref{eq:hgeofun-HGODE}, we have 
\begin{prop}
    \label{prop:hgeofun-cont-relations}
 One can read all $15$ relations among the contiguous functions  by equating
 any two lines of the right hand
 side of \normalfont{Eq.} \eqref{eq:hgeofun-cont-relations}: 
\begin{align}
    \begin{split}
        z \frac{dF}{dz} &= z \frac{ab}{c} F(a+,b+,c+) \\
        &=a(F(a+) -F) \\
        &=b(F(b+) -F) \\
        &=(c-1) (F(c-) -F) \\
        &= \frac{
            (c-a) F(a-) + (a-c+bz)F
        }{ 1-z}  \\
&= \frac{
            (c-b) F(b-) + (b-c+az)F
        }{ 1-z} \\
        &= z \frac{
            (c-a)(c-b) F(c+) + c(a+b-c)F
        }{ c(1-z)}.
    \end{split}
    \label{eq:hgeofun-cont-relations}
\end{align}
\end{prop}
There are other type of symmetries, for example, fractional linear transformations:
\begin{align}
    \begin{split}
        \pFq21(a,b;c;z)& = (1-z)^{-b} \pFq21(c-a,b;c;\frac{z}{z-1}), \\
        \pFq21(a,b;c;z) &= (1-z)^{-a} \pFq21(a,c-b;c;\frac{z}{z-1}),
    \label{eq:hgeofun-Pfaff-1}
    \end{split}
\end{align}
which are known as Pfaff transformations. 
Then the Euler transformation 
\begin{align}
    \label{eq:hgeofun-Euler}
    \pFq21(a,b;c;z)&= (1-z)^{c-a-b} \pFq21(c-a,c-b;c;z)
\end{align}
follows quickly.
The following initial values can be obtained without much calculation:
\begin{align}
    \pFq21(1,1;2;z) = -\ln(1-z)/z, \,\,\,
    \pFq21(a,b;b;z) = (1-z)^{-a}
\label{eq:hgeofun-specialcases}.
\end{align}
In particular, we see that $\pFq21(1,1;2;1-z)$ agrees with
the modified-log function 
$\mathcal L_0(z)$  in the first proof of the Gauss-Bonnet on
$\T^2_\theta$ \cite{MR2907006}.

%

Let us transfer Eqs \eqref{eq:hgfun-diff-relations-GF21-abc+1} to
    \eqref{eq:hgfun-diff-relations-GF21-c-1} to the family $H_{a,b}$ 
following  the change of notations given in \cref{eq:H-alpha-to-FD^(n)}.  
\begin{prop}
    For the one variable family $H_{a,b}(z;m)$ 
    defined in \cref{eq:H-alpha--m-j=2},
    the following functional
    relations hold:
    \begin{align}
        H_{a,b}(z;m+2) &= (\tilde d_m+ zd/dz) 
        H_{a,b}(z;m) ,
        \label{eq:hggeofun-DR-tildeKanm+2}
        \\
        H_{a,b+1}(z;m) &= (b^{-1}d/dz) H_{a,b}(z;m) ,
        \label{eq:hggeofun-DR-tildeKanb+1}
        \\
        H_{a,b+1}(z;m) &= (1+b^{-1}zd/dz) H_{a+1,b}(z;m) 
    \label{eq:hggeofun-DR-tildeKanb+1-a+1},
    \end{align}
    where $\tilde d_m \defeq \tilde d_m(a,b)=a+b+m/2-2$.
    Moreover,
    \begin{align}
\label{eq:hggeofun-DR-tildeKan-GCF}
        \frac{ a+b}{ \tilde d_m} \frac{
          H_{a+1,b}(z;m)
        }{H_{a,b}(z;m)}  = 
        \frac{
            \pFq21(\tilde d_m+1,b;a+b+1;z)
        }{
        \pFq21(\tilde d_m,b;a+b;z)
    }.
    \end{align}
    The right hand side is a Gauss's continued fraction.
\end{prop}
We see from \cref{eq:hggeofun-DR-tildeKanb+1} that the computation of $H_{a,b}$
is reduced to $H_{a,1}$. 
We also mention that there are already 111951 formulas implemented
in \textsf{Mathematica}
\footnote{see \url{http://functions.wolfram.com/HypergeometricFunctions/Hypergeometric2F1/}}
with regard to the evaluation of $\pFq21$.

Let us move on to the two variable family $F_D^{(2)} = F_1$ known as Appell's
$F_1$ functions: 
\begin{align*}
    &\,\, \frac{\Gamma(b) \Gamma(b') \Gamma(c-b-b')}{\Gamma(c)} 
    F_1(a;b,b',c;z_1,z_2) \\
    =& \,\,
    \int_0^1 \int_0^{1-t} u^{b-1} v^{b'-1} (1-u-v)^{c-b-b'-1}
    (1-z_1 u - z_2 v)^{-a} dudv  .  
\end{align*}
It is a solution of the PDE system:
\begin{align}
    \sbrac{
    x(1-x) \partial_x^2 + y(1-x) \partial_x \partial_y
    + [c-(a+b+1)] \partial_x - by \partial_y -ab
}F_1 &=0,
    \label{eq:hgeofun-} \\
 \sbrac{
    y(1-y) \partial_y^2 + x(1-y) \partial_x \partial_y
    + [c-(a+b'+1)] \partial_x - b'x \partial_y -ab'
}F_1 &=0.
\end{align}
As before, we set $F_1 \defeq F_1(a;b,b',c;x,y)$
and $F_1(a+) \defeq F_1 ( a+1;b,b',c;x,y)$,  same pattern applies to $F_1(b+)$,
$F_1(b'+)$ and $F_1(c+)$, then
\begin{align}
    \partial_x F_1 = F_1(a+,b+,c+), \,\,
    \partial_y F_1 = F_1(a+,b'+,c+), 
    \label{eq:hgeofun-f1abc+-diffsys} 
\end{align}
also
\begin{align}
  &  F_1(a+) =   a^{-1}(a+x\partial_x + y\partial_y)F_1 
    \label{eq:hgeofun-f1a+-diffsys} \\
  &F_1(b+) =  b^{-1}( b+x\partial_x )F_1 , \, \, \, \, 
    F_1(b'+) =  b'^{-1}( b'+y\partial_y )F_1 
    \label{eq:hgeofun-f1b'+-diffsys} \\
  &F_1 =   c^{-1}(c+ x \partial_x + y\partial_y )F_1(c+) 
\label{eq:hgeofun-f1c+-diffsys}.
\end{align}
The following special reduction of $F_1$ functions are easy to verify:
\begin{align}
    \label{eq:hgfun-F1toGF21-2}
    \begin{split}
    F_1(a;b,b';c;0,y) &= \pFq21(a,b';c;y), \\
    F_1(a;b,b';c;x,0) &= \pFq21(a,b;c;x).    
    \end{split}
    \end{align}
In addition,
\begin{align}
    F_1(a;b,b';c;x,x) &= (1-x)^{c- a-b-b'} 
    \pFq21(c-a;c-b-b';c;x) \nonumber\\
    &= \pFq21(a,b;b+b';x),  \nonumber\\
    F_1(a;b,b';b+b';x,y) &= (1-y)^{-a} \pFq21(a,b;b+b';\frac{x-y}{1-y}).
    \label{eq:hgfun-F1toGF21-1}
\end{align}
Similar to the Pfaff transformation 
\eqref{eq:hgeofun-Pfaff-1}, we have:  
\begin{align}
    \begin{split}
    F_1(a;b,b';c;x,y) &= 
    (1-x)^{-b} (1-y)^{-b'} F_1(c-a;b,b';c;\frac{x}{x-1},\frac{y}{y-1}) \\
    &=(1-x)^{-a}  F_1(a;c-b-b',b';c;\frac{x}{x-1},\frac{y-x}{1-x})
    .
    \end{split}
    \label{eq:appF1-F1syms}
\end{align}

\begin{prop}
    Let $\tilde d_m \defeq \tilde d_m(a,b,c)=a+b+c+m/2-2$.  
    For the two variable family $H_{a,b,c}(u,v;m)$ defined in
    \cref{eq:H-alpha--m-j=2}, we have
    \begin{align}
    H_{a,b,c}(u,v;m+2) &= (\tilde d_m + u\partial_u +v\partial_v)
    H_{a,b,c}(u,v;m),
        \label{eq:hpgeofun-DR-Habcm+2} \\
        H_{a,b+1,c}(u,v;m) &= b^{-1} \partial_v H_{a,b,c}(u,v;m),    
        \label{eq:hpgeofun-DR-Habc-b+1} \\
        H_{a,b,c+1}(u,v;m) &= c^{-1} \partial_u H_{a,b,c}(u,v;m) .   
        \label{eq:hpgeofun-DR-Habc-c+1} 
    \end{align}
    When increasing the parameter $a$ by one, we encounter a similar ration as in
\eqref{eq:hggeofun-DR-tildeKan-GCF}:
    \begin{align*}
        \frac{
        H_{a+1,b,c}(u,v;m)
    }{ H_{a,b,c}(u,v;m) } = 
    \frac{\tilde d_m}{ a+b+c}
    \frac{
        F_1(\tilde d_m+1;c,b;a+b+c+1;u,v)
    }{
        F_1(\tilde d_m;c,b;a+b+c;u,v) 
    } .
    \end{align*}
\end{prop}
\begin{proof}
    See Eqs \eqref{eq:hgeofun-f1abc+-diffsys} to \eqref{eq:hgeofun-f1c+-diffsys}.
\end{proof}

In the early stage of the project, the author found the following reduction of
$F_1$ to $\pFq21$ making use of the Appell's $F_2$ functions:
\begin{align*}
        &\,\, \frac{\Gamma(b) \Gamma(b') \Gamma(c-b)
            \Gamma(c'-b')
        }{\Gamma(c) \Gamma(c')
        } 
        F_2(a;b,b';c,c';x,y ) \\
    =& \,\,
    \int_0^1 \int_0^{1} u^{b-1} v^{b'-1} (1-u)^{c-b-1}
(1-v)^{c'-b'-1}
    (1-x u - y v)^{-a} dudv  .  
\end{align*}
First, we have
\begin{align}
\begin{split}
     F_1(a,b,b',c;x,y) &= (x/y)^{b'} F_2 (b+b';a,b';c,b+b',x,1-x/y) \\
     &= (y/x)^{b} F_2(b+b';a,b;c,b+b',y,1-y/x)       ,
    \end{split}
    \label{eq:hgeofun-APF2toAPF1}
\end{align}
cf. \cite[Sec.  16.16]{MR2723248} or \cite[Sec. 5.10, 5,11]{MR0058756}.

\begin{prop}[\cite{MR2107356}, Theorem 2]
    For $a \in \mathbb{C}$, $b \in \mathbb{C}\setminus \Z_{\le 0}$,  
    $p,q \in \Z_{\ge 0}$,  $p<q$ 
    and $\abs x + \abs y <1$,
    \begin{align}
       \begin{split}
        &\,\, F_2(q+1,a,p+1;b,p+2;x,y) \\
        =&\,\,
        -\frac{p!}{q(1-q)_p} \frac{p+1}{y^{p+1}} 
        \pFq21 (a,q-p;b;x)  \\ 
        + & \frac{p+1}{y^{p+1}} \sum_0^p \frac{(-1)^{k}}{(q-k) (1-y)^{q-k}}     
        \binom pk \sum_0^{p-k} (-x)^{m} \binom{p-k}{m} \frac{ (a)_m}{(b)_m} \\
        \cdot & \pFq21 \brac{a+m,q-k;b+m; \frac{x}{1-y} } .\\
        \end{split}
        \label{eq:hpgeofun-F2toGF21-thm}
    \end{align}
\label{prop:hpgeofun-F2toGF21-thm}
\end{prop}
Although the formula above is lengthy,
the case with $p=0$ and $q=1$ is quite enlightening
\begin{align}
    \label{eq:F_1-as-didiff-2F1}
    F_1(a;1,1;c;x,y) = \frac{
        x \pFq21(a,1;c;x)  - y \pFq21(a,1;c;y)  
    }{x-y},
\end{align}
which led us to the multivariable generalizations, cf. 
\S\ref{subsec:reduction-to-GF21}.

We end the discussion on Gauss and Appell's hypergeometric functions 
with a small application of the differential and contiguous
relations.
Formulas in \cref{prop:tildeKab-tildeHabc} has been proved 
in the author's previous work
\cite[Eq. (3.9)]{LIU2017138} which gives partial solutions to the computation
of $H_\alpha(\bar z;m)$ when $m$ is even. In fact, what is behind  is a 
recursive relations in the dimension parameter $m$.
\begin{prop}
        \label{prop:hgfun-relations-K-H} 
    Let $m=\dim M \ge 2$ and $a,b,c$ be positive integers,
    \begin{align}
        H_{a,b}(z;m+2) &= a H_{a+1,b}(z;m)
        + bH_{a,b+1}(z;m)
        \label{eq:hgfun-relations-K-m+2-a+1-b+1} \\
        \begin{split}
         H_{a,b,c}(u,v;m+2) &= a H_{a+1,b,c}(u,v;m)
 +b H_{a,b+1,c}(u,v;m) \\ 
 &+  c H_{a,b,c+1}(u,v;m).
   \end{split}
                \label{eq:hgfun-relations-H-m+2-a+1-b+1} 
    \end{align}
\end{prop}
\begin{proof}
    Notice that
    \eqref{eq:hgfun-relations-K-m+2-a+1-b+1} and
    \eqref{eq:hgfun-relations-H-m+2-a+1-b+1} are equivalent to the following
    contiguous relations  respectively:
  \begin{align}
\begin{split}
       &\,\,  (a+b) \pFq21(\tilde d_m+1,b;a+b;u)\\
        =&\,\,
        a \pFq21(\tilde d_m+1,b;a+b+1;u) +
        b \pFq21(\tilde d_m+1,b+1;a+b+1;u),           
        \end{split}
        \label{eq:hgeofun-evenm-tildeKab-inductive-step-v2}
    \end{align}
and
    \begin{align}
        \label{eq:hgeofun-evenm-tildeHabc-inductive-step-v2}
        \begin{split}
            &\,\, (a+b+c) F_1(\tilde d_m +1;c,b;a+b+c;u,v) 
        \\ =&\,\, 
        a F_1(\tilde d_m +1;c,b;a+b+c+1;u,v) \\
        +& \,\, b F_1(\tilde d_m +1;c,b+1;a+b+c+1;u,v) \\
        +& \,\,
        c F_1(\tilde d_m +1;c+1,b;a+b+c+1;u,v).
        \end{split}
    \end{align}
    We prove \eqref{eq:hgeofun-evenm-tildeHabc-inductive-step-v2} as an example
    and left \eqref{eq:hgeofun-evenm-tildeKab-inductive-step-v2} to interested
    readers.
Let $F = F_1(\alpha
    ;\beta,\beta';\gamma;u,v)$, according to  
\cref{eq:hgeofun-f1a+-diffsys,eq:hgeofun-f1b'+-diffsys,eq:hgeofun-f1c+-diffsys},
    we can solve for
    $(u\partial_u+v\partial_v)F$ in two different ways:
 \begin{align*}
     (u\partial_u+v\partial_v)F &= (\gamma-1) (F(\gamma-) - F) \\
&= 
\beta(F(\beta+)-F) + \beta'(F(\beta'+)-F).
    \end{align*}
    Equation \ref{eq:hgeofun-evenm-tildeHabc-inductive-step-v2} follows as a special
     instance with parameters:
     $\alpha = \tilde
     d_m+1$, $\beta = c$, $\beta'=b$ and $\gamma = a+b+c+1$. 
\end{proof}

\begin{prop}
        \label{prop:tildeKab-tildeHabc} 
Assume that the dimension $m$ is even and greater or equal than $4$,
    for $j_m =(m-4)/2 \in \Z_{\ge0}$ and non-negative integers $a,b,c$, we have:
    \begin{align}
        H_{a,b}(u;m) &= \frac{d^{j_m}}{dz^{j_m}} \Big|_{z=0}
        (1-z)^{-a} (1-u-z)^{-b} ,
        \label{eq:hgeofun-evenm-tildeKab} \\
        H_{a,b,c}(u,v;m) &=
        \frac{d^{j_m}}{dz^{j_m}} \Big|_{z=0}
        (1-z)^{-a} (1-u-z)^{-b}(1-v-z)^{-c}.
        \label{eq:hgeofun-evenm-tildeHabc} 
    \end{align}
\end{prop}
\begin{proof}
    We only prove \cref{eq:hgeofun-evenm-tildeKab} as an example since the
    arguments are quite similar.
    When $m =4$, $\tilde d_m = a+b$,
    \cref{eq:hgeofun-evenm-tildeKab} follows from the identity:
    \begin{align*}
        \pFq21(\alpha,\beta;\beta,z) = \pFq21(\beta,\alpha;\beta,z)
        = (1-z)^{-\alpha}.
    \end{align*}
    Now assume that \cref{eq:hgeofun-evenm-tildeKab} holds for some even $m$,
    following the induction argument, we need to show that 
    \begin{align}
        H_{a,b}(u;m+2)& = \frac{d^{j_m}}{dz^{j_m}} \Big|_{z=0}
        \frac{d}{dz}\sbrac{ (1-z)^{-a} (1-u-z)^{-b} }\nonumber \\
        &= a H_{a+1,b}(u;m) + b H_{a,b+1}(u;m),    
        \label{eq:hgeofun-evenm-tildeKab-inductive-step}
    \end{align}
    which is one of the recurrence relations in \cref{prop:hgfun-relations-K-H}.

\end{proof}


\subsection{Divided Differences}
\label{subsec:divided-diff-notation}
We recall basic properties of divided difference in order to generalize   
\cref{eq:F_1-as-didiff-2F1}. 
For a one-variable function $f(z)$, the divided difference can be defined
inductively as below:
\begin{align*}
    f[x_0] & \defeq f(x_0) ;\\
    f[x_0, x_1, \dots, x_n] & \defeq 
    \left( 
        f[x_0, \dots, x_{n-1}] - f[x_1, \dots, x_{n}]
    \right)/ (x_0 - x_n)
\end{align*}
For example, 
\begin{align*}
    f[x_0,x_2] = (f(x_0) - f(x_1)) / (x_0 - x_1). 
\end{align*}
One can show in general: 
\begin{align}
    f[x_0, x_1, \dots, x_n] = \sum_{l=0}^n  
    f(x_l)  \prod_{s=0, s \neq l}^n (x_l -x_s)^{-1}.
    \label{eq:divided-gen-formula}
\end{align}
There is also a residue version:
\begin{align}
    \label{eq:ddiff-residue-version}
    f[z_1, \dots, z_n] = \frac{1}{2 \pi i} \int_{\tilde C}
    \frac{f(z)}{
        (z-z_1) \cdots (z-z_n)
    } dz,
\end{align}
where the contour circles around the points $z_1, \dots, z_n$ exactly once.

When $f$ is a  multivariable function, we  use a subscript to indicate the
variable on which the divided difference acts, for example, 
$f(x,z,y)[z_1, \dots, z_s]_z$. 
Through out the paper, we fix the variable $z$  as the default choice for the
divided difference operator: $[\bullet, \dots, \bullet] \defeq [\bullet,
\dots, \bullet]_z$.

Basic properties include:
\begin{enumerate}
    \item Symmetry: $f[x_0, \ldots,x_n]$ is a symmetric function, that is, one
        can freely permute the variables $x_0, \ldots, x_n$.  
    \item The Leibniz rule:
        \begin{align}
            \label{eq:dfiff-Leibniz}
            f(x_0) g[x_0, \ldots, x_n]
            (fg)[x_0, \ldots, x_n] = 
            f(x_0) g[x_0, \ldots, x_n] +
            g(x_0) f[x_0, \ldots, x_n]
            .
        \end{align}
    \item Composition rule:
        \begin{align}
            (f[y_1, \ldots, y_q,z])[x_1, \ldots,x_n]_z =
            f[y_0, \ldots, y_q, x_1, \ldots,x_n].
        \end{align}
    \item The confluent case: suppose  there are $a+1$ identical copies of $x$
        in the arguments, we define 
        \begin{align}
            f[y,x,\ldots,x] = \frac{1}{a!} \partial^a_x f[y,x].    
        \end{align}
\end{enumerate}

\subsection{Reduction to Gauss Hypergeometric Functions}
\label{subsec:reduction-to-GF21}

The reduction formula \cref{eq:F_1-as-didiff-2F1} admits a more elementary
proof that works for general Lauricella functions   $F_D^{(n)}$. 
It begins with another integral representation (compared with 
the one variable case in \cref{eq:hgeo-defn-pFq21}) of $F_D^{(n)}$:
    \begin{align}
        \label{eq:FD-int-rep-from-0-1}
      & \,\, F_D^{(n)}(a; \alpha_1, \dots, \alpha_n;c; z_1, \dots, z_n) 
        \nonumber \\
        =& \,\,
        \frac{\Gamma(c)}{\Gamma(a)\Gamma(c-a)}
        \int_0^1 t^{a-1} (1-t)^{c-a-1} (1-z_1 t)^{-\alpha_1} \cdots
        (1-z_nt)^{-\alpha_n} dt,
\end{align}
where we require $\Re c > \Re a$ for the converges of the integral. 
The essential observation is the following divided difference relation between
the following two factors in 
\cref{eq:hgeo-defn-pFq21,eq:FD-int-rep-from-0-1}, respectively:
\begin{align}
    \label{eq:pFq21-factor-ddiff-FD-factor}
    (z^{n-1}(1-z t)^{-1}) [ z_1, ,,,z_n] = (1-z_1 t)^{-1} \cdots (1-z_n t)^{-1}.
\end{align}
\begin{prop}
        \label{prop:FD-GF21-divideddiff}
        For $a\in \mathbb{C}$ and $c \in \mathbb{C} \setminus \Z_{\le 0}$,
        we have
    \begin{align}
        F_D^{(n)}(a,1,\dots,1,c; z_1, ,,,z_n) = 
        (z^{n-1} \pFq21(a,1,c))[ z_1, ,,,z_n]_z
        \label{eq:FD-GF21-divideddiff}
    \end{align}
\end{prop}
\begin{proof}
When  $\Re c > \Re a$, \eqref{eq:FD-GF21-divideddiff} follows from 
\cref{eq:pFq21-factor-ddiff-FD-factor}.
To remove the restriction 
 $\Re c > \Re a$, it sufficient to show that
\eqref{eq:FD-GF21-divideddiff} is preserved with respect to the differential
relation:
\begin{align*}
     F(c) &= c^{-1}(c+z\partial_z) F(c+), \\
    F_D^{(n)}(c) &=  
    c^{-1}(c+z_1\partial_{z_1} + \cdots + z_n \partial_{z_n})F_D^{(n)}(c+),
\end{align*}
 where, on the right hand side,
 we have borrowed the $c\pm$ notation from the discussion of contiguous
 relations of the Gauss hypergeometric functions, to indicate adding or
 subtracting $1$ regarding to the parameter $c$,   
see for instance, \cref{eq:hgfun-diff-relations-GF21-a+1}.
 The underlying functions are
\begin{align*}
    F = \pFq21 (a,1;c;z), \, \, 
    F_D^{(n)}(c) = F_D^{(n)}(a,1,\ldots,1;c;  z_1, \ldots, z_n).
\end{align*}
It follows  that 
\cref{eq:FD-GF21-divideddiff} holding
for some $c$
implies that it holds for $c-1$ as well. To be more specific, let us assume 
\cref{eq:FD-GF21-divideddiff} holds for some $c+1$, we claim that
\begin{align*}
    F_D^{(n)}(c) &=  
      c^{-1}(c+z_1\partial_{z_1} + \cdots + z_n \partial_{z_n})
    (z^{n-1} F(c+;z)) [ z_1, ,,,z_n]_z \\
    &=
    (z^{n-1} ( c^{-1}(c+z\partial_z))F(c+;z)) [ z_1, ,,,z_n]_z \\
    &= (z^{n-1} F(c;z)) [ z_1, ,,,z_n]_z, 
\end{align*}
in which the non-trivial step is to show that 
\begin{align}
    P_n (z^{n-1} F(c+;z)) [ z_1, ,,,z_n]_z =  
    (z^{n-1} P_1 F(c+;z)) [ z_1, ,,,z_n]_z, 
    \label{eq:claim-P_n-P_1}
\end{align}
where $P_n$ and $P_0$ stand for the following differential operators:
\begin{align*}
    P_n = z_1\partial_{z_1} + \cdots + z_n \partial_{z_n}, \,\,
    P_1 = z\partial_z
    .
\end{align*}
To shorten the proof, 
we make use of the residue version of divided difference given in 
\cref{eq:ddiff-residue-version}. 
One can of course use \cref{eq:divided-gen-formula} to complete the argument
which does not require $f$ to be analytic. 
Observe that $z_l \partial_{z_l} (1-z_l/z)^{-1}
= -z\partial_z (1-z/z_l)^{-1}$, furthermore:
\begin{align*}
    P_n \brac{ \frac{1}{(1-z_1/z) \cdots (1- z_n/z)}  }
    = - P_1 \brac{ \frac{1}{(1-z_1/z) \cdots (1- z_n/z)}  }.
\end{align*}
 We pass the differentials through the
contour integral and then apply integration by parts to reach 
\cref{eq:claim-P_n-P_1}:
\begin{align*}
 &\, \,    P_n (z^{n-1} F(c+;z)) [ z_1, ,,,z_n]_z
    =  
\frac{1}{2 \pi i} \int_{\tilde C}
P_n \brac{
\frac{ z^{n-1} F(c+;z)}{
        (z-z_1) \cdots (z-z_n)
 } } dz \\
 =&\, \, 
\frac{1}{2 \pi i} \int_{\tilde C}
z^{-1} F(c+;z) 
P_n \brac{
\frac{1}{
        (1-z_1/z) \cdots (1-z_n/z)
 } } 
  \\
 =&\, \,   
\frac{1}{2 \pi i} \int_{\tilde C}
z^{-1} F(c+;z) 
 (-z\partial_z) \brac{
\frac{1}{
        (1-z_1/z) \cdots (1-z_n/z)
 } }
  \\
 =&\, \,  
\frac{1}{2 \pi i} \int_{\tilde C}
\frac{\partial_z F(c+;z) }{
        (1-z_1/z) \cdots (1-z_n/z)
 }
 = 
 \frac{1}{2 \pi i} \int_{\tilde C}
\frac{ z^{n-1} (z\partial_z) F(c+;z) }{
        (z-z_1) \cdots (z-z_n)
 }
 \\
    = &\, \, 
    (z^{n-1} P_1 F(c+;z)) [ z_1, ,,,z_n]_z .
\end{align*}
We add the assumption that 
$c \in \mathbb{C} \setminus \Z_{\le 0}$ because the inductive argument fails
for $c=0$. 
\end{proof}

In conclusion, we have obtained a complete reduction of $H_{\alpha}(\bar
z;m;j)$ to the one variable family $H_{a,1}(z;m;j)$ by iterated
divided differences and differentials.
\begin{thm}
    \label{thm:red-to-GF21}
    For the spectral functions in the $V_j$-term computation, we have the
    following two-step reduction to the Gaussian hypergeometric functions    
    $H_{a,1}(z;m;j)$ generated by differentials and divide differences
    respectively: 
    \begin{align}
        H_{\alpha_0, \alpha_1, \dots,\alpha_n}
            (z_1, \dots , z_n;m;j)  &=
            \frac{ 
                \partial_{z_1}^{\alpha_1-1} \cdots
\partial_{z_n}^{\alpha_n-1}
}{
    (\alpha_1-1)!  \cdots (\alpha_n-1)!
} H_{\alpha_0, 1, \dots,1} 
(z_1, \dots , z_n;m;j),
            \label{eq:Ha0an-to-Ha011} \\
        H_{a,1, \dots , 1}(z_1, \dots, z_n;m;j) 
        &=(z^{n-1} H_{a+n-1,1}(z;m;j))[z_1, \dots, z_n ]_z   
            \label{eq:Ha011-to-Ha01-divdiff} 
    \end{align}
\end{thm}
\begin{rem}
      We compare the result to 
      \cite[\S4,\S5]{leschdivideddifference}, in which the spectral
      functions are generated in the same fashion by the  modified logarithm
      function $\mathcal L_0$ (first observed in
      \cite{MR2907006}), which, in particular,  is a  hypergeometric function: 
\begin{align*}
    H_{1,1}(1-z;2) = \pFq21(1,1;2;1-z) = \frac{\ln z}{z-1}
    = \mathcal L_0(z).
\end{align*}
\end{rem}
\begin{proof}
    The differential reduction \eqref{eq:Ha0an-to-Ha011} follows from the
    differential relations:
    \begin{align*}
        F_D(\alpha_q +) = (\alpha_q)^{-1} \partial_{z_q} F_D , \,\,\,
        1\le q \le n,
    \end{align*}
    where $F_D = F_D(a;\alpha_1, \dots , \alpha_n; z_1, \dots, z_n)$.

   The second equation 
    \eqref{eq:Ha011-to-Ha01-divdiff} follows from \cref{eq:FD-GF21-divideddiff}
    according to the change of parameters \cref{eq:H-alpha-to-FD^(n)}.
\end{proof}

\section{ 
    Pseudo-differential Approach to Heat Trace Asymptotic on $\T^m_\theta$ 
}
\label{sec:symbloic-computation}

\subsection{Notations for $\T^m_\theta$}

For any skew-symmetric matrix $\theta = ( \theta_{ij}) \in M_{m \times
m}(\R)$, the  corresponding noncommutative $m$-torus $C(\T^m_\theta)$ is
the universal $C^*$-algebra generated by $m$ unitary elements $U_1, \ldots,
U_m$, that is
\begin{align*}
    U_j U_j^* = U_j^* U_j =1, \, \, \, \, j=1, \ldots,m,
\end{align*}
subjected to the relation
\begin{align*}
    U_i U_j = e^{2 \pi i \theta_{ij}} U_j U_i, \, \, \, \, 
    1\le  i,j \le  m.
\end{align*}
The smooth noncommutative torus $C^\infty(\T^m_\theta)$ is 
a dense subalgebra consisting of elements of the form
\begin{align}
    \label{eq:a=sum-abarl-defn-notations}
    a = \sum_{\bar l \in \Z^m} a_{\bar l} U^{\bar l}, \, \,  l=(l_1, \ldots, l_m),
    \, \,  U^{\bar l} \defeq U_1^{l_1} \cdots U_m^{l_m},
\end{align}
where the coefficients $a_{\bar l} \in  \mathbb{C}$ are of rapid decay in $\bar
l$, namely, 
$\abs{ a_{\bar l} P(\bar l)}$ is bounded for any polynomial $P(\bar l)$ in
$\bar l$. 
There is a tracial
functional $\varphi_0$ implemented by  taking the zero-th coefficients
in the series \cref{eq:a=sum-abarl-defn-notations}: 
\begin{align}
    \label{eq:varphi_0-defn-notations}
    \varphi_0: C^\infty(\T^m_\theta) \to  \mathbb{C}: \, \, 
    \varphi_0(a) = 
    \varphi_0 \brac{ \sum_{\bar l \in \Z^m} a_{\bar l} U^{\bar l}}
    \defeq a_0.
\end{align}
It gives rise to a inner product
$\abrac{\cdot ,\cdot } \defeq \abrac{\cdot , \cdot }_{\varphi_0}$
\begin{align*}
    \abrac{a,b} = \varphi_0(b^* a), \, \, \forall a,b \in  C^\infty(\T^m_\theta).
\end{align*}
The associated GNS-representation (of $C^\infty(\T^m_\theta)$) is the Hilbert
obtained as the completion of  the smooth functions with respect to the inner
product
\begin{align}
    \label{eq:cal-H_0-defn}
  \mathcal H_0 \defeq \overline{C^\infty(\T^m_\theta)}^{\abrac{\cdot ,\cdot }} ,  
\end{align}
which plays the role of $L^{2}$-functions on $\T^m_\theta$ .

As shown in \cref{eq:a=sum-abarl-defn-notations}, $C^\infty(\T^m_\theta)$
admits a $\Z^m$ grading  that induces a $\T^m$ action by demanding $\{U^{\bar
l}\}_{\bar l \in  \Z^m}$ to be the eigenvectors. At the infinitesimal level,
the natural basis of the Lie algebra $\R^m$  gives rise to the basic derivations
that generate the algebra of differential operators. 
\begin{defn}
    \label{defn:basic-derivation-del-defn}
    The basic derivations $\delta_1, \ldots, \delta_m$ on
    $C^\infty(\T^m_\theta)$ described above can be defined on 
generators of $C^\infty(\T^m_\theta)$ as below 
\begin{align}
    \label{eq:basic-derivation-del-s}
    \delta_s (U_j) = \mathbf 1_{sj} U_j, \, \, \, \, 
    1\le  s,j \le m,
\end{align}
and then extended to  the whole algebra by the Leibniz property. 
Here $\mathbf 1_{sj}$ stands for the Kronecker-delta symbol\footnote{
The symbol $\delta$ has been taken by the basic derivations.}.
In later computation, we shall also use $\nabla_l \defeq -i \delta_l$ which,
from deformation perspective,
corresponding to the partial derivative $\partial_{x_l}$ or the covariant
derivative along $\partial_{x_l}$: 
$\nabla_{\partial_{x_l}} = \partial_{x_l}$, where $\nabla$ is the flat
connection and $(x_1, \ldots, x_m)$ is the coordinate system on
$\T^m \cong \R^m/ (2\pi  \Z)^m$.  
\end{defn} 


\subsection{Symbols of Pseudo-differential Operators}
\label{subsec:symbols-of-pseudo-ops}
We assume the reader's acquaintance with Connes's
pseudo-differential calculus \cite{connes1980c} and notes by Baaj
\cite{MR967366,MR967808}. 
We mention some  recent papers with new contributions to the contraction of
pseudo-differential calculi:
\cite{MR3540454} extends the
contraction to pseudo-differential operators acting on bundles (Heisenberg
modules); \cite{MR3985230,MR3985231} give a detailed description of the 
theory on all noncommutative $m$ tori; the author \cite{Liu:2015aa,LIU2017138} 
achieves one on toric
noncommutative manifold by applying the $\theta$-deformation machinery 
onto Widom's intrinsic pseudo-differential calculus on smooth manifolds.

With the refers listed above,
we only review how to compute the symbol of differential operators and the
notion of order for pseudo-differential operators or their symbols, which are
sufficient to under the algorithm of computing heat coefficients.

Roughly speaking,   the symbols are coming from
Fourier transforms of pseudo-differential operators that belong to
$ C^\infty(\R^m, C^\infty(\T^m_\theta))$ viewed as smooth functions on the
cotangent bundle on $\T_\theta^m$.  
Let $\xi = (\xi_1, \ldots, \xi_m)$ be a coordinate system on $\R^m$, then  
the symbols of the basic derivations are mapped to the coordinate functions:
$\sigma(\delta_l) = \xi_l$,
 $1\le l \le m$. Furthermore, for a general differential operator $P = \sum
 a_\alpha \delta^\alpha$, with $\alpha$ runs over a finite collection of
 multiindices in $\Z_{\ge  0}^{m}$, the symbol $\sigma(P) = \sum a_\alpha
 \xi^{\alpha}$ is polynomial in $\xi$ with coefficients $a_\alpha \in
 C^\infty(\T^m_\theta)$.

 The notion of parametric symbols 
 $ S\Sigma (\R^m \times  \Lambda, C^\infty(\T^m_\theta))$
 is required in the pseudo-differential approach to the heat trace asymptotic.
The conic region $\Lambda \subset \mathbb{C}$ serves as the domain of the
resolvent parameter $\lambda$ in \cref{eq:heat-op-holo-cal-defn}.
Elements $p(\xi,\lambda)$ are functions smooth in
 $\xi$ and analytic in $\lambda$ with value in $C^\infty(\T^m_\theta)$.
 Both the algebra of pseudo-differential operators and symbols admit
 a filtration which extends the notion of  \emph{order} of differential
 operators,
\begin{align*}
    S\Sigma
    = \bigcup_{d=-\infty}^\infty S\Sigma^d
    \subset C^\infty(\R^m \times \Lambda, C^\infty(\T^m_\theta)) .
\end{align*}
We will only encounter homogeneous symbols that are representatives of the
corresponding graded algebra:  
$\oplus_{d=-\infty}^\infty S\Sigma^d / S\Sigma^{d-1}$.  
Homogeneity of  degree $d \in  \Z$ is defined as below: 
\begin{align}
    \label{eq:homogeneity-d-symbols}
    p(c \xi, c^2 \lambda) = c^d p(\xi, \lambda), \, \, \, \, 
    \forall  c>0 .
\end{align}
We remark that the resolvent parameter $\lambda$ is labeled as  a degree $2$
symbol because
only the heat trace of second order differentials operators  are concerned in
the paper. 
For a non-parametric symbol (i.e. independent of $\lambda$), say 
a polynomial symbol in $\xi$ (of a differential operator), 
the homogeneity in \cref{eq:homogeneity-d-symbols}
reflects the order of differentiation. For instance, the symbol of the flat
Laplacian is identical to the classical counterpart
\begin{align*}
    \sigma(\Delta) = \abs\xi^2 \defeq \sum_{l=1}^m \xi_l^2.
\end{align*}
The pseudo-differential symbols appeared in later computations
(see \cref{eq:b_j-finite-sum}) are  generated by 
polynomial symbols and 
the resolvent symbol of order $-2$ :
\begin{align*}
    b_0 (\xi, \lambda) = (e^h \abs\xi^2 - \lambda)^{-1}, \, \, \, \, 
    h = h^* \in  C^\infty(\T^m_\theta).
\end{align*}
\subsection{Metric Structures}
\label{subsec:metric-structure}
For the metric structure,
we recall on the commutative two torus, the Riemann metrics can be parametrized
by $(\tau,h)$ in which the modulus $\tau \in  \mathbb{C}$ with $\Im \tau >0$
indicates the conformal class (through the correspondence between complex
and conformal structures), 
while $h \in C^\infty(\T^2, \R)$ is a smooth real value function whose
exponential $k = e^h$ is the Weyl
factor parametrizing the conformal change of metrics.  

The analogue on $C^\infty(\T^2_\theta)$ was implemented as a modular spectral
triple
in \cite[\S1]{MR3194491}. 
We recall only the associated Laplacian operators acting on functions.
Let us fix the conformal class  $\tau = \sqrt{-1}$ whose
 flat (Dolbeault) Laplacian is given by $\Delta = \delta_1^2 + \delta_2^2$.  
For a Weyl factor $k = e^h$ with self-adjoint $h \in  C^\infty(\T^2_\theta)$, 
the new Laplacian under the conformal change of metric is given by 
\begin{align}
    \label{eq:Del-varphi-defn}
  \Delta_\varphi = k^{\frac{1}{2}} \Delta k^{\frac{1}{2}}: 
  \mathcal H_0\to \mathcal H_0 ,
\end{align}
where the Hilbert space $\mathcal H_0$ is defined in \cref{eq:cal-H_0-defn}. 
The construction of $\Delta_\varphi$ is modeled on the Dolbeault Laplacian
$\bar \partial^*_\varphi \bar \partial$, in which $\bar \partial$ is untouched
since we are in the same conformal class while $\bar \partial^*_\varphi$
indicates that the adjoint is taken with respect to a new inner product 
\begin{align}
    \label{eq:varphi-defn}
    \abrac{a,b}_\varphi \defeq \varphi(b^* a), \, \, \, \,  \forall  a,b 
    \in  C^\infty(\T^2_\theta),
\end{align}
and $\varphi$ is the rescaled volume weight: 
\begin{align*}
    \varphi(a) \defeq \varphi_0(a  k^{-1} ), \, \,  \forall  
    a \in  C^\infty(\T^2_\theta) .
\end{align*}
The word ``modular'' comes from the modular theory for von Neumann algebras.
The new volume functional is only a weight with the following 
KMS (Kubo-Martin-Schwinger) property: 
\begin{align}
    \label{eq:KMS-cond-defn}
    \varphi(a b) = \varphi \brac{b \mathbf y(a)}, \, \, 
    \mathbf y(a) \defeq \mathbf y_h(a) = k^{-1}(a) k, \, \, \, \, 
\forall a \in C^\infty(\T^2_\theta) .
\end{align}
In \cite{MR3194491}, 
$\mathbf y$ and its logarithm $\mathbf x  \defeq \mathbf x_h = \log \mathbf
y =[\cdot , h]$ are called the \emph{modular operator} and  the 
\emph{modular derivation} respectively, because they are a generator (resp.
infinitesimal generator) of the one-parameter
group of modular automorphisms $\mathbf y^{-it}$, $t \in  \R$ 
of the weight $\varphi$.

The Riemannian geometry on higher dimensional noncommutative tori is still an
open area for research. Ponge recently proposed a construction for general
Riemannian metrics and Laplace-Beltrami operators \cite{ha2019laplace}.
Nevertheless, the main focus of the paper is to investigate the universal
structures of the heat coefficients. 
For general noncommutative $m$-tori $C^\infty(\T^m_\theta)$,  
we shall directly take the result from dimension two, that is to simply view  
\begin{align*}
    \Delta \to  \Delta_\varphi = k^{\frac{1}{2}} \Delta k^{\frac{1}{2}}
\end{align*}
as a conformal change of metric $g_0 \to  g = k g_0$ 
with respect to the Weyl factor $k$, where 
 \begin{align}
     \label{eq:flat-Laplacian-m}
  \Delta = \delta_1^2 + \cdots + \delta_m^2   
 \end{align}
is the flat Laplacian  representing the standard Euclidean  metric $g_0$.

\subsection{Heat Trace Asymptotic and Local Invariants}
\label{subsec:heat-trace-asym-local-invariants}

After identifying the metric with a geometric differential operator, in our
case the Laplacian $\Delta_\varphi$, we would like to study the small time
asymptotic heat trace functional: $a \in  C^\infty(\T^m_\theta)$: 
 \begin{align}
     \label{eq:small-time-heat-asym-functional}
    a \to  \Tr \brac{ a e^{-t\Delta_\varphi}} \sim_{t \searrow 0}
    \sum_{j=0}^\infty V_j(a , \Delta_\varphi) t^{(j-m)/ 2} ,
\end{align}
whose coefficients yield the corresponding local invariants. 
The locality refers to the fact that the functionals $a \to  V_j(a
, \Delta_\varphi)$ are absolutely continuous to the flat volume functional
$\varphi_0$ (cf. \cref{eq:varphi_0-defn-notations}) with Radon-Nikodym
derivative $R^{(j)}_{\Delta_\varphi} \in  C^\infty(\T^m_\theta)$ (also referred
as functional densities in the paper): 
\begin{align}
    \label{eq:densities-V_j-defn}
    V_j(a, \Delta_\varphi) = \varphi_0 \brac{ a R^{(j)}_{\Delta_\varphi} },
    \, \,  
    \forall  a \in  C^\infty(\T^m_\theta) .
\end{align}

The analogy is made   
 the spectral geometry of Riemannian manifolds. Let $\Delta$ be the scalar
 Laplacian operator on a closed Riemannian manifold $(M,g)$. Same expansion as
 in \cref{eq:small-time-heat-asym-functional} holds, in which all the odd
 coefficients are zero and for even $j$, the functional density consists of
 local invariants that can be written as polynomials in the components of the
 $j$-th derivative of  metric tensor. In particular, upto a universal
 constant depending on $m$,  the $V_0$-term 
 coefficient  captures the volume functional and the $V_2$-term   reveals the
 scalar curvature function $\mathcal S_g \in  C^\infty(M)$ of the metric: 
\begin{align*}
    V_0(f, \Delta) = \int_M f d\mu_g, \, \, \, \, 
    V_2(f, \Delta) = \int_M f \brac{\mathcal S_g / 6} d\mu_g,
    \, \, \, \,  \forall  f\in  C^\infty(M). 
\end{align*}

Therefore, we shall call the functional density of 
$R^{(2)}_{\Delta_\varphi} \in  C^\infty(\T^m_\theta)$
in \cref{eq:densities-V_j-defn} the modular curvature of $\Delta_\varphi$. 
The word ``modular'' indicates the significant role of the modular
operator/derivation (see \cref{eq:KMS-cond-defn}) in the final result of 
$R^{(2)}_{\Delta_\varphi}$. Also, we do not specify 
$R^{(2)}_{\Delta_\varphi}$ as ``modular scalar curvature'' because
the model of conformal change of metric is not strong enough at the current stage. 

%

In the theory of elliptic operators on  manifolds, the heat family with
$t>0$ of an elliptic operator $P$ can be defined using
holomorphic functional calculus:
\begin{align}
    \label{eq:heat-op-holo-cal-defn}
    e^{-t P}  = \frac{1}{ 2 \pi i} \int_C e^{-t\lambda} (P- \lambda)^{-1}
    d\lambda,
\end{align}
where $C$ is a suitable contour winding around the spectrum of $P$. 
Same formula holds  for $ P = \Delta_\varphi$, 
provided the fact that $\Delta_\varphi$ an unbounded
operator on the Hilbert space $\mathcal H_0$ (cf. \cref{eq:cal-H_0-defn}) with
discrete spectrum contained in $[0,\infty)$.  

Two obstacles are presented when looking at the heat trace $\Tr(e^{-tP})$
in terms of \cref{eq:heat-op-holo-cal-defn}: 
\begin{enumerate}[i)]
    \item the resolvent $(P- \lambda)^{-1}$ involving the inverse of
a differential operator;
\item a trace formula for pseudo-differential operators.
\end{enumerate}

\subsection{Resolvent Approximation}


The pointwise multiplication between symbols does not
reflects the composition of the associated pseudo-differential operators.
The heart of a symbol calculus is the construction of a formal $\star$ product
which doest the job. 
Let $P$ and $Q$ be two pseudo-differential operators
with symbol $p$ and $q$ respectively. Then the symbol of their composition has
a formal expansion
\begin{align*}
    \sigma(PQ) =  p \star q \backsim
    \sum_{j=0}^\infty a_j(p,q),
\end{align*}
where each $a_j(\cdot,\cdot)$ is a bi-differential operator such that $a_j(p,q)$
reduces the total degree by $j$ and ``$\sim$'' means modulo lower degree terms.

The $\star$-product is not canonical. On manifolds, any choice of a connection
yields a construction. In Connes's pseudo-differential calculus, the formulas
of $a_j$ shares the same structure as those coming from a flat connection on
manifolds: 
\begin{align}
    a_j(p,q) = \frac{(-i)^j}{j!} (D^j p) \cdot (\nabla^j q)
    \defeq
    \frac{(-i)^j}{j!}
    \sum_{\mu_1, \ldots, \mu_j=1}^m
    (D^j p)_{\mu_1 \cdots \mu_n}
    (\nabla^j q)_{\mu_1 \cdots \mu_n}
    ,   
    \label{eq:b2cal-a_j-flatconnection}
\end{align}
where $D = D_\xi$ 
is the vertical differential while $\nabla =-i \delta$ is the horizontal
differentials implemented by the basic derivations (cf.
\cref{defn:basic-derivation-del-defn}). We borrow the tensor notations in
differential geometry: the two rank $j$ tensors 
$ D^j p$ and $\nabla^j q$ 
 are contravariant and covariant  respectively, so that the contraction
 $\cdot$ applies. Their  components are given by partial derivatives 
\begin{align*}
    (D^j p)_{\mu_1 \cdots \mu_n} = 
    D_{\xi_{\mu_1}} \cdots  D_{\xi_{\mu_n}} (p), 
\, \, \, \, 
(\nabla^j p)_{\mu_1 \cdots \mu_j} = (-i)^j
\delta_{\mu_1} \cdots \delta_{\mu_j} (q) .
\end{align*}

Let $P$ be an second order elliptic operator acting on $C^\infty(\T^m_\theta)$.
From the discussion in \S\ref{subsec:symbols-of-pseudo-ops}, 
the symbol is a polynomial of degree two with the decomposition:
\begin{align}
    \label{eq:sym-P=p_2+p_1+p_0}
    \sigma(P -\lambda) (\xi, \lambda) = p(\xi) = p_2 (\xi,\lambda) 
    + p_{1} (\xi) + p_0 (\xi)
\end{align}
with $p_j$ is of degree $j$ in $\xi$, $j=0,1,2$ and 
$p_2 (\xi,\lambda) = p_2(\xi) - \lambda$. Notice that $\lambda$ is of degree
$2$  according to \cref{eq:homogeneity-d-symbols} thus get matched with  the
leading term of $\sigma(P)$.
Suppose that the resolvent approximation is of the form
\begin{align*}
    \sigma \brac{ (P-\lambda)^{-1}} \sim 
    b_0(\xi,\lambda) + b_1(\xi,\lambda) + \cdots ,
\end{align*}
where $b_j(\xi,\lambda)$ is of degree $-j-2$, $j=0,1,2,\ldots$.  
They are subject to the equation  
    \begin{align}
        \label{eq:bjsum-star-pmu=1}
        (b_0+b_1+ \cdots) \star \sigma(P - \lambda) \backsim 1,
    \end{align}
    which can be further decomposed 
by grouping terms with the same  homogeneity $i=0,-1,-2,\ldots$: 
\begin{align*}
    \sum_{i=0}^{-\infty}
    \brac{
    \sum_{
    -2-l-\mu+\nu = i 
}
    a_\mu(b_l, p_\nu ) 
    }
    \sim 1,
\end{align*}
where $\nu = 0,1,2$ and $\mu,l \in  \Z_{\ge 0}$.  
For $i=0$, we have only one term $a_0(b_0,p_2) = b_0 p_2 = 1$, the ellipticity
of $P$ implies the existence of the resolvents:     
\begin{align*}
    b_0 (\xi,\lambda) = (p_2 (\xi) - \lambda)^{-1} \in  
    C^\infty(\R^m \times  \Lambda,C^\infty(\T^m_\theta)). 
\end{align*}
One can recursively construct the rest of $b_j$, $j=1,2,\dots$, by looking at the
corresponding equations for $i=-1,-2, \ldots$. 
\begin{align}
    \label{eq:b_j-general-form}
    b_j = \brac{
        \sum_{
        \substack{
            l+\mu - \nu +2= j
    \\ 0\le l\le j-1, \, \,  0\le \nu \le 2
        }
        } a_\mu(b_l, p_\nu)
    } 
    (-b_0).
\end{align}
We shall make use of only the first two terms in the paper:
\begin{align}
    \label{eq:b2cal-b1term-gen}
 b_1 &= [ \mathit{a}_0\left(b_0,p_1\right)+\mathit{a}_1\left(b_0,p_2\right)](-b_0),
\\
 b_2 &= 
[\mathit{a}_0\left(b_0,p_0\right)+\mathit{a}_0\left(b_1,p_1\right)
+\mathit{a}_1\left(b_0,p_1\right)+\mathit{a}_1\left(b_1,p_2\right)
+\mathit{a}_2\left(b_0,p_2\right)]
(-b_0).
\label{eq:b2cal-b2term-gen}
\end{align}
The number of summands of $b_j$  increases dramatically as $j$ goes up.    
Nevertheless, they are conceptually simple. 
After expanding  the bi-differential operators  $a_\mu(p_\nu,
b_l)$,  $b_j$'s  are finite sums (could be very lengthy\footnote{ 
    For $b_4$-term, we quote from \cite{2016arXiv161109815C}:
``the process of calculating this term involves exceedingly lengthy expressions
and at times involves manipulations on a few hundred thousand terms'' 
}) of the form:   
\begin{align}
    \label{eq:b_j-finite-sum}
   b_j = \sum b_0^{\alpha_0} \rho_1 b_0^{\alpha_1} \rho_2
   \cdots b_0^{\alpha_{n-1}}
    \rho_n b_0^{\alpha_n} ,
\end{align}
where the exponents $\alpha$'s are non-negative integers and the $\rho \defeq
\rho(x)$'s are the derivatives of the symbols of $P$, namely,
$p_\nu$, $\nu=0,1,2$  (in particular, no inverse operation involved and
independent of $\lambda$). 



\subsection{The Application of the Rearrangement Lemma}
\label{subsec:application-of-rearr-lem}
We reach the heat trace via a trace formula that computes the operator
trace of a trace-class pseudo-differential operators in terms of its symbol.
For example, \cite[Theorem 6.1]{MR3540454} 
is one of the  parametric version of such results designed for proving
heat trace asymptotic.
\begin{prop}
    The heat trace asymptotic \cref{eq:small-time-heat-asym-functional} exists, 
    in which the $j$-th heat coefficient is determined by  
the resolvent approximation $b_j$ in the following way:
\begin{align}
    \label{eq:hgfun-bj-to-Vj}
    V_j(a, P) = \varphi_0 \brac{
        a  \left[ \int_{\R^m} \frac{1}{2\pi i }  
            \int_C e^{-\lambda}
    b_j(\xi, \lambda) d\lambda d\xi \right] 
    } , \;\; \forall a \in C^\infty(\T^m_\theta).
\end{align}
In other words, the functional densities are given by
\begin{align}
    \label{eq:hgfun-bj-to-Rj}
    R^{(j)}_P =
\int_{\R^m} \frac{1}{2\pi i }  
            \int_C e^{-\lambda}
    b_j(\xi, \lambda) d\lambda d\xi .
\end{align}
\end{prop}
\begin{proof}
    The result has been proved in a more general setting: on Heisenberg modules
    \cite[\S6]{MR3540454} and on toric noncommutative manifolds
    \cite[\S6]{Liu:2015aa}. They share the same blueprint from 
 \cite[\S1.7]{gilkey1995invariance}. 
 Moreover, the arguments in \cite{branson1986conformal} can also be adapted
 into the noncommutative setting which shows in addition that one can perform
 term by term differentiation in the heat trace asymptotic 
 \cref{eq:small-time-heat-asym-functional} when consider suitable variational
 problems of the heat trace. This property has been used in \cite{MR3194491} and
 \cite{LIU2017138}.
\end{proof}

The new input to the pseudo-differential approach 
is a rearrangement process for summands shown in 
\cref{eq:b_j-general-form}. Since the $\rho$'s has no dependence on $\lambda$,
we can factor them out by making use of the rearrangement operators discussed in 
\S\ref{subsec:smooth-funcal-A-tensor-n-times}:
\begin{align*}
  b_0^{\alpha_0} \rho_1 b_0^{\alpha_1} \rho_2
   \cdots b_0^{\alpha_{n-1}}
    \rho_n b_0^{\alpha_n}  
    =
    (b_0^{\alpha_0} \otimes  \cdots  \otimes  b_0^{\alpha_n}) \cdot  
    (\rho_1 \otimes  \cdots  \otimes  \rho_n).
\end{align*}
The function $G_\alpha(\bar s)$ in \cref{eq:Galpha-from-contour-int} is exactly
the scalar version of the operator valued integral
\begin{align*}
     \frac{1}{2\pi i} \int_C e^{-\lambda} 
(b_0^{\alpha_0} \otimes  \cdots  \otimes  b_0^{\alpha_n}) d\lambda.
\end{align*}

Now we assume that the leading symbol $p_2 = k \abs\xi^2$ where $k =e^h$,
$h=h^* \in C^\infty(\T^m_\theta)$ is  a Weyl factor. We shall perform
integration in $\xi$  using spherical coordinates 
$\xi = (s,r)$ and 
$d\xi = d \mathbf s (r^{m-1}dr)$,  where $s$ and 
$d \mathbf s$ are the coordinate and the measure on the unit sphere $S^{m-1}$
while   $r \ge  0$ is the radial  coordinate. 
Notice that $b_0$ is independent of $s$, thus spherical integration
concentrates on the $\rho$-factors of the form: 
$\rho_l (\xi) = \rho_l (r,s) = r^{ \gamma_l } \rho_l(s)$, where $\gamma_l$ is the
polynomial degree in $\xi$, $1\le l\le n$.  
After integrating over $S^{m-1}$ , it becomes  
\begin{align*}
    \int_{S^{m-1}} 
    \rho_1 (\xi) \otimes  \cdots  \otimes  \rho_n(\xi) d \mathbf s =
    r^d \bar \rho \in  C^\infty(\T^m_\theta)^{\otimes n}.
\end{align*}
Since $b_j$ (in \cref{eq:b_j-finite-sum})is of degree $-2-j$, 
the power $d = \sum \gamma_l $  subjects to 
\begin{align}
    \label{eq:d-and-alpha-eq-homog-of-bj}
    -2 \abs\alpha +d =-2-j, \, \, \, \,  \text{with} \, \, 
    \abs\alpha = \sum_{l=0}^n \alpha_l.
\end{align}
Notice the origin is no long a singularity of the contour integral 
after replacing the resolvent $(P-\lambda)^{-1}$  by $b_j(\xi,\lambda)$. 
In particular, we can
take the contour $C$ in \cref{eq:hgfun-bj-to-Rj,eq:hgfun-bj-to-Vj}
to be the imaginary axis oriented as $-i\infty$ to
$i\infty$. Set $\lambda = i x$ with $x \in  \R$ and 
denote 
$\bar \rho = \tilde \rho_1 \otimes  \cdots  \otimes  \tilde \rho_n \in
C^\infty(\T^m_\theta)^{\otimes n}$. 
Recall $p_2(\xi)
= k\abs\xi^2$ so that the resolvent reads 
\begin{align*}
    b_0(\xi,\lambda)= b_0(r,x) = (k r^2 - ix)^{-1}     .
\end{align*}
We further expand $R_P^{(j)}$ in  \cref{eq:hgfun-bj-to-Rj} according to 
$b_j$ given in \cref{eq:b_j-finite-sum}
\begin{align*}
    R_P^{(j)} & = \sum 
    \sbrac{
    \frac{1}{2\pi } \int_0^\infty \int_{-\infty}^\infty 
    e^{-i x} 
    (b_0^{\alpha_0} \otimes  \cdots  \otimes  b_0^{\alpha_n})
    r^{d+m-1}
    dx dr
    }
    (\bar \rho) \\
              &=
              \frac{1}{2} \sum
              \sbrac{
\frac{1}{2\pi } \int_0^\infty \int_{-\infty}^\infty 
    e^{-i x} 
    \prod_{l=1}^n (k^{(l)} r - ix)^{-\alpha_l}
    r^{( d+m-2 ) /2} 
    dx dr 
              }
              (\bar \rho)
    ,
\end{align*}
where $d = 2\abs\alpha -2-j$ satisfies \cref{eq:d-and-alpha-eq-homog-of-bj}
above and the overall
factor $1 /2$ comes from the substitution $r \to  r^2$.
The rearrangement operator in the squared brackets is exactly of the form given in 
\cref{thm:rarr-lem}.  Sum up, we have obtained the following general form of
the heat coefficients.
\begin{prop}
    \label{prop:R^j_P-general-form}
   Assume that the leading symbol of the elliptic operator $P:
   C^\infty(\T^m_\theta)\to  C^\infty(\T^m_\theta)$  is of the form $p_2
   = k\abs\xi^2$ with a Weyl factor $k=e^h$, $h=h^* \in  C^\infty(\T^m_\theta)$,
   the functional density of the $j$-the  heat trace asymptotic 
   coefficient consists of finite sums  of the form:
   \begin{align}
    \label{eq:R^j_P-general-form}
      R^{(j)}_P  = \frac{1}{2}
       \sum 
       k^{-( d(\alpha,m,j)+1)}
       \tilde H_\alpha(\mathbf z; d(\alpha,m,j)) (\bar \rho)
   \end{align}
   where $\bar \rho \in  C^\infty(\T^m_\theta)^{\otimes n}$, 
   $\alpha = (\alpha_0, \ldots, \alpha_n) \in  \Z_{+}^{n+1}$ with 
   $\abs\alpha = \sum_0^n \alpha_l$ and
   \begin{align}
       \label{eq:d(alpha-m-j)=}
       d(\alpha,m,j) = \abs\alpha + (m-j) /2 -2.
   \end{align}
    \end{prop}
\begin{rem}
    \hspace{1cm}
    \begin{enumerate}
        \item We have  suppressed the notation for functional calculus, that is 
            $\tilde H_\alpha$ means $( \tilde H_\alpha)_\gamma$ (cf.
            \cref{eq:f-gamma=Psi-gamma(f)}).
            We will keep the abbreviation in the rest of the paper and
            use bold font letters $\mathbf x$, $\mathbf y$ and $\mathbf
            z$ (cf. \cref{eq:mathbf-z^(n)-defn}) to emphasize the operator
            natural of the variables.
        \item The length $n$ can vary in the finite sum.    
        \item The overall factor $1/ 2$  comes from the change of
            variable $r \to  r^2$. 
    \end{enumerate}
\end{rem}

We conclude  the section with a comparison to  the argument given 
in \cite[\S6]{MR3194491} for handling  the integral in \cref{eq:hgfun-bj-to-Rj}.
 It started with
switching the order of integration between $d\lambda$ and $dr$, which 
requires the integrability of $r^m b_j(r,\lambda)$ in $r$.  
Due to the specific homogeneity
of $b_2$,  the result of contour integral is simply setting
$\lambda =-1$. 
Note that differentiating in $\lambda$ reduces the homogeneity of $b_j$ by $2$.
When $r^m b_j(r,\lambda)$ is not integrable in $r$, 
one can use integration by parts to replace it by 
$r^m \partial_\lambda^\nu b_j(r,\lambda)$, whose integral in $r$  converges
when $\nu$ is large.  
The precise statement (cf. \cite[Lemma 6.1]{LIU2017138}) reads
\begin{align}
\int_0^\infty
\label{eq:set-lamda=-1}
    \frac{1}{2\pi i }  
            \int_C e^{-\lambda}
            b_j(r, \lambda) r^{m-1} d\lambda dr =
\int_0^\infty
\sbrac{
    \partial_{\lambda}^{\nu_0}\big |_{\lambda=-1} 
b_j(r,\lambda)} r^{m-1} dr,
\end{align}
where $\nu_0 =  (m-j) /2$ provided $b_j$ is of homogeneity $-j-2$. Namely,
$\nu_0$ is the smallest integer such that the homogeneity of
$\partial_{\lambda}^{\nu_0} b_j$  is strictly less than $-m$.
Notice that $b_j$ already contains large
number of summands as in \cref{eq:b_j-general-form}.
The differential $\partial_\lambda^{\nu_0}$ will, in principle, double,
triple, $\ldots$, the number as $m$ going up. In special cases, we have concise
formulas  like those \cref{prop:tildeKab-tildeHabc}, but the dependence on $m$ 
does not provides  satisfactory solutions 
regarding to   applications in  modular geometry, such as 
\cref{thm:-funrel-all-m}.

So far, we have finished the general discussion of the pseudo-differential
approach to the heat coefficients and would like to provide a complete
calculation of a $V_2$-coefficient in the next section.

\subsection{The $V_2$-term of $\Delta_k = k \Delta$ }
Let $k =e^h$ with self-adjoint $h \in  C^\infty(\T^m_\theta)$ be a Weyl factor
as usual. Denote by $\mathbf y$ and $\mathbf x$ the corresponding modular
operator and modular derivation. 
In the section, we shall carry out explicit computation of the $V_2$-term of
the perturbed  Laplacian $\Delta_\varphi$ which represents 
curved metric inside the conformal class of the flat one
(cf. \S\ref{subsec:metric-structure}).

To simplified the computation, it is better to  work with
\begin{align}
    \label{eq:Del-k-defn}
    \Delta_k = k \Delta = k^{\frac{1}{2} } \Delta_\varphi k^{-\frac{1}{2}}  
: C^\infty(\T^m_\theta) \to C^\infty(\T^m_\theta) 
\end{align}
whose symbol decomposition is given by
 $\sigma(\Delta_k) = \sigma(k \Delta) = p_2 + p_1 +p_0$ with
\begin{align*}
    p_2 = k \abs\xi^2, \,\,\, p_1 =p_0 =0.
\end{align*}
Their heat trace  functionals are almost identical: $\forall  a \in
C^\infty(\T^m_\theta)$: 
\begin{align*}
    \Tr \brac{a e^{-t\Delta_\varphi}} = 
    \Tr \sbrac{
    \mathbf y^{-\frac{1}{2}}
    \brac{a e^{-t\Delta_\varphi}}  
    } =
    \Tr \brac{ \mathbf y^{-\frac{1}{2}}(a) e^{-t\Delta_k}} . 
\end{align*}
It follows that 
\begin{align*}
    R^{(j)}_{\Delta_\varphi} = \mathbf y^{\frac{1}{2}}
    \brac{ R^{(j)}_{\Delta_k}}, \, \, 
    j=0,1,2\ldots.
\end{align*}
The merit of computing $R^{(j)}_{\Delta_k}$ first  
is clear:
the vanishing of $p_1$ and $p_0$ provides a considerable cut
to the number of summands of the final expansion of $b_2$.

Recall that the vertical differential $D = \partial_\xi$ are partial
derivatives in $\xi \in  \R^m$  while the horizontal differential 
$\nabla  = -i\delta$ are proportional to the basic derivations in
\cref{eq:basic-derivation-del-s}. 
As preparation for expanding the bi-differential operators $a_j(\cdot, \cdot)$ 
in \cref{eq:b2cal-a_j-flatconnection}, we compute: 
\begin{align}
    \begin{split} 
    (D p_2)_j = 2 k \xi_j, \,\, & (D^2 p_2)_{jl} = 2 k \mathbf 1_{j l}, 
    \\
    (\nabla p_2)_j
    = r^2 (\nabla k)_j  , \,\,\,& (\nabla^2 p_2)_{jl} = r^2 (\nabla^2
    k)_{jl} ,
    \end{split}
    \label{eq:b2cal-vetandhordiffofp2}
\end{align}
where $r = \abs\xi$ and  
$\mathbf 1_{jl}$ stands for the $(j,l)$-entry of the identity matrix (of
size $m \times  m$).

In \cref{eq:bjsum-star-pmu=1}, we put the polynomial symbols $p_\nu$
($\nu=0,1,2$) on the right on purpose, so that $b_j$ only appears in the first
slots of the bi-differential operators $a_j(\cdot ,\cdot )$. 
As a result, we have successfully avoid the occurrence of $\nabla^\alpha b_0$   
in the general form of $b_j$ in \eqref{eq:b_j-general-form},
The reason is 
the commutativity:
\begin{align*}
    [p_2, (D p_2)_j] =0, \, \, \, \, 
    [p_2, (\nabla  p_2)_j] \neq 0.
\end{align*}
As a consequence, components of $D^l b_0$  behave normally: 
\begin{align*}
    (D b_0)_j &= D_j (p_2 - \lambda)^{-1} = - b_0^{2} (D_j p_2)\\
    (D^2 b_0)_{jl} &= 2b_0^3 (D p_2)_j (D p_2)_l -b_0^2 (D^2 p_2)_{jl},
\end{align*}
For $\nabla  b_0$, only the Leibniz property survives
\begin{align*}
    0= \nabla_j( p_2 b_0) = p_2 (\nabla_j b_0) + (\nabla_j p_2) b_0,
\end{align*}
thus $(\nabla  b_0)_j = - b_0 (\nabla  p_2)_j b_0$.
One   repeats the Leibniz rule to 
reach higher order derivatives $\nabla^\alpha b_0$.  
The number of summands is much larger compared to the
vertical differential $D^\alpha b_0$.



By carefully expanding the general formulas of $b_1$ and $b_2$ given in
\cref{eq:b2cal-b1term-gen,eq:b2cal-b2term-gen}, we obtain:
\begin{align}
\begin{split}
    b_2(\xi,\lambda) = & \,\,  4 r^2 \xi _j \xi _l b_0^3.k^2. (\nabla^2 k)_{l,j}.b_0-r^2. \mathbf 1_{j l}. b_0^2.k. (\nabla^2 k)_{l,j}.b_0 \\
 & \,\,  + 4 r^2 \xi _j \xi _l b_0^2.k. (\nabla k)_l.b_0. (\nabla k)_j.b_0 
 -4 r^4 \xi _j \xi _l b_0^2.k. (\nabla k)_l.b_0^2.k. (\nabla k)_j.b_0,
\\
&\,\,
+ 2 r^4 .\mathbf 1_{j l} .b_0^2.k.(\nabla k)_l.b_0.\nabla (k)_j.b_0-8 r^4 \xi _j \xi _l b_0^3.k^2.(\nabla k)_l.b_0.(\nabla k)_j.b_0,
\end{split}
\label{eq:b2cal-b2termwithxi}
 \end{align}
 where the summation is taken over repeated indices from $1$ to $m$.

 \begin{lem}
     Let $\xi =(\xi_1, \ldots, \xi_m)$ and 
     $\xi = (r,s) \in (0,\infty) \times S^{m-1}$ be the Cartesian and the
     spherical coordinates of $\R^m$ respectively. 
     We recall the volume of the unit sphere $S^{m-1}$ 
     \begin{align}
         \label{eq:volume-Sm-1}
         \mathrm{Vol}(S^{m-1}) &=
         \int_{S^{m-1}}    d\mathbf{s} = \frac{2 \pi^{m/2}}{\Gamma(m/2)}.
     \end{align}
     For monomials, the integration is given by: 
     \begin{align}
         \label{eq:int-Sm-1-xi-(jl)}
         \int_{S^{m-1}} \xi_j \xi_l (r,s)   d\mathbf{s} &=
         \frac{\pi^{m/2}}{\Gamma(1+m/2)} \mathbf{1}_{jl} r^2
         = \mathrm{Vol}(S^{m-1}) \frac1m \mathbf 1_{jl} r^2.
     \end{align}
 \end{lem}
 \begin{proof}
     The integration of the monomials  \cref{eq:int-Sm-1-xi-(jl)} follows from 
     the symmetry of the sphere,  we have: 
    \begin{enumerate}
        \item for $i \neq j$, 
    complete cancellation occurs so that the result is zero.
\item for $i=j$,  
    \begin{align*}
        \int_{S^{m-1}} \xi_j^2 (r,s)   d\mathbf{s} =
   \frac{1}{m}     \int_{S^{m-1}} \sum_{l=1}^m \xi_l^2 (r,s)   d\mathbf{s} =
   \frac{1}{m}     \int_{S^{m-1}}  r^2  d\mathbf{s} 
   = \frac{r^2}{m} \mathrm{Vol}(S^{m-1}).
    \end{align*}
    \end{enumerate}
 \end{proof}
 After applying the lemma above to 
 \cref{eq:b2cal-b2termwithxi}, we see that,
 upto an overall factor $\mathrm{Vol}(S^{m-1})$, $\tilde b_2(r,\lambda)$,  the
 integration of $  b_2(\xi,\lambda)$ over the unit sphere is given by
 \begin{align}
 \begin{split}
 &\,\,
  (\mathrm{Vol}(S^{m-1}))^{-1} \tilde b_2(r,\lambda)  
 \defeq 
 (\mathrm{Vol}(S^{m-1}))^{-1}
 \int_{S^{m-1}} b_2 (\xi,\lambda) d \mathbf s
     \\ 
 =&\, \, 
 -r^2 \mathbf 1_{j l}. b_0^2.k.(\nabla^2 k)_{l,j}.b_0 + \frac{4 r^4 \mathbf 1_{j
 l} b_0^3.k^2.(\nabla^2 k)_{l,j}.b_0}{m} \\
 &\,\,
 +2 r^4 \mathbf 1_{j l}. b_0^2.k.(\nabla k)_l.b_0.(\nabla k)_j.b_0 +
 \frac{4 r^4 \mathbf 1_{j l}. b_0^2.k.(\nabla k)_l.b_0.(\nabla k)_j.b_0}{m} \\
 &\,\, 
 -\frac{8 r^6 \mathbf 1_{j l}. b_0^3.k^2.(\nabla k)_l.b_0.(\nabla k)_j.b_0}{m}
 -\frac{4 r^6 \mathbf 1_{j l}. b_0^2.k.(\nabla k)_l.b_0^2.k.(\nabla
 k)_j.b_0}{m}   , 
 \end{split}
\label{eq:b2cal-b_2inr}
 \end{align}
 After summing over repeated indices, we see that the 
 $\tilde \rho$ in \cref{eq:R^j_P-general-form}
  will be one of the following  forms:
    \begin{align}
        \label{eq:Tr-nabla^2k-nablak-otimes-nablak}
        \Tr(\nabla^2 k) = \sum_{l=1}^m \nabla^2_l k = -\Delta k,
        \, \, 
        \Tr(\nabla k \otimes  \nabla  k) 
        = \sum_{l=1}^m (\nabla_l k) \otimes  (\nabla_l k) .
    \end{align}
\begin{thm}
    \label{thm:psecal-R_Delta_k}
    Denote by $R_{\Delta_k} \defeq R^{(2)}_{\Delta_k} \in C^\infty(\T^m_\theta)$ the
    functional density of the second heat coefficient of $\Delta_k = k \Delta$.  
    Upto an overall constant $\op{Vol}(S^{m-1})/2$, 
    \begin{align}
    &\, \,  (\op{Vol}(S^{m-1})/2)^{-1}
        R_{\Delta_k} 
        \nonumber \\
        =&\, \, 
        k^{-\frac{m}{2}} 
        K_{\Delta_k} (\mathbf y ;m) \brac {\Tr(\nabla^2 k)}
        + k^{-\frac{m}{2}-1}
        H_{\Delta_k}(\mathbf y^{(1)}, \mathbf y^{(2)} ;m) 
        \brac{
        \Tr(\nabla k \otimes  \nabla k) 
        }
    \label{eq:psecal-R_Delta_k}
    \end{align}
    where the rearrangement operators $K_{\Delta_k} (\mathbf y ;m)$  
    and $H_{\Delta_k}(\mathbf y^{(1)}, \mathbf y^{(2)} ;m)$ are introduced in 
    $\S\ref{subsec:smooth-funcal-A-tensor-n-times}$ with respect to the underlying
    spectral functions
\begin{align}
     K_{\Delta_k} (y;m) = \frac4m H_{3,1}(z;m) - H_{2,1}(z;m),
     \label{eq:b2cal-KDeltak}
 \end{align}
 with $z=1-y$ and 
 \begin{align}
\label{eq:b2cal-HDeltak}
\begin{split}
      H_{\Delta_k} (y_1 , y_2 ;m)  &=
      (\frac{4 }{m}+2) H_{2,1,1}\left(z_1,z_2;m\right)  -\frac{4 (1-z_1)
     H_{2,2,1}\left(z_1,z_2 ;m\right)}{m}
     \\ &  
     -\frac{8
     H_{3,1,1}\left(z_1,z_2;m\right)}{m},
\end{split}
    \end{align}
    where $z_1 = 1- y_1$ and $z_2 = 1-y_1 y_2$.
\end{thm}
\begin{proof}
   We have shown that
   \begin{align*}
       R_{\Delta_k} 
       = (2\pi)^{-1} \int_0^\infty \int_{-\infty}^\infty
    \tilde   b_2(r,ix) dx (r^{m-1} dr) 
   \end{align*}
   and explained in \S\ref{subsec:application-of-rearr-lem} that how to apply
 \cref{thm:rarr-lem} to compute the integration, cf.
 \cref{prop:R^j_P-general-form}.
Let us carry out one example, 
the one  with $H_{2,2,1}$  which has an extra 
  factor $(1-z_1)$ shown in \cref{eq:b2cal-HDeltak}:
\begin{align*}
     & \,\, 
     (2\pi)^{-1} \int_0^\infty  \int_{-\infty}^\infty
     e^{-ix}
\mathbf 1_{j l}. b_0^2.k.(\nabla k)_l.b_0^2.(\nabla k)_j.b_0
dx (r^{m-1} dr)  
\\ = & \,\,
k^{-(m/2+3)}
H_{2,2,1}(\mathbf z^{(1)}, \mathbf z^{(2)} ;m)
\brac{
   k (\nabla_l k) \otimes k (\nabla_l k) 
}\\
= &\, \, 
k^{-(m/2+1)}
\mathbf y^{(1)} H_{2,2,1}(\mathbf z^{(1)}, \mathbf z^{(2)} ;m)
\brac{
    (\nabla_l k) \otimes  (\nabla_l k) 
},
 \end{align*}
 where $\mathbf y^{(1)} = 1 - \mathbf z^{(1)}$
 is responsible for moving the $k$ in the second factor
 of the tensor to the very left.
Instead of repeating the process, we just point out that 
   the function $K_{\Delta_k}$ collects the contribution from the first two
   terms in \cref{eq:b2cal-b_2inr} while  $H_{\Delta_k}$ is determined by the last
   four terms.
\end{proof}

\section{Functional Relations}
\label{sec:FunRel}


In regard to the modular geometry described in
\S\ref{subsec:metric-structure} and \ref{subsec:heat-trace-asym-local-invariants},
geometric consequences are reflected as internal relations between the spectral
functions $K_{\Delta_k}$ and $H_{\Delta_k}$. 
In the following two examples, we provide new simplifications by making use of the
differential  and recursive relations of the hypergeometric family
$H_\alpha(\bar z;m)$.

 \subsection{Gauss-Bonnet Theorem for $\T^2_\theta$}
 \label{subsec:GB-thm}
 Let us fix $m=2$ in this section \S\ref{subsec:GB-thm}.
 Recall from \cite{MR3194491,MR2907006} that,
 with regard to the model of conformal change of metric: $\Delta \to
 \Delta_\varphi = k^{\frac{1}{2}} \Delta k^{\frac{1}{2}}$, 
 with Weyl factor $k = e^h$, $h =h^* \in  C^\infty(\T^m_\theta)$. The functional 
 \begin{align*}
     h \to  F(h) = F(k) \defeq \zeta_{\Delta_\varphi}(0)
 \end{align*}
recaptures the integration of the Gaussian curvature on Riemann surfaces.
Therefore the Gauss-Bonnet theorem can be rephrased in terms of the spectral
zeta functions: $\zeta_{\Delta_\varphi}(0) = \zeta_{\Delta}(0)$, that is, the
value of the spectral zeta function at zero is independent of  the metrics
(parametrized by the Weyl factors $k$) and agrees with the commutative 
counterpart (because the flat Laplacian $\Delta$ is isospectral to the one on
flat two torus).

 It has been shown that functional above, upto a constant, is equal to the
 integration of the modular curvature:
 $F(k) = \varphi_0(R_{\Delta_\varphi}) =
 \varphi_0( R_{\Delta_k})$, 
 and admits a local expression 
\begin{align}
    \label{eq:funrel-Gauss-Bonnet-localform}
    F(k) = 
    \sum_{l=1,2}
    \varphi_0\brac{
        k^{-2} T(\mathbf y)(\nabla_l k) (\nabla_l k) 
    } ,
\end{align}
where $T(\mathbf y):C^\infty(\T^m_\theta) \to  C^\infty(\T^m_\theta)$ is
a rearrangement operator via the function 
\footnote{
    The function $T(y)$ agrees with (upto some constant) the $f(u)$ in
    \cite[Lemma 3.2]{MR2907006}. 
}
\begin{align}
    T(y) = - K_{\Delta_k}(1) \frac{y^{-2} -1}{y-1} + H_{\Delta_k}(y, y^{-1}).
    \label{eq:funrel-T}
\end{align}
Notice that with respect to the local form \cref{eq:funrel-Gauss-Bonnet-localform},
the Gauss-Bonnet theorem becomes that $F(k)$ is the zero functional in $k$ becase
 $F(1) = 0$.

Let $\tilde T(y) = y^{-2} T(y^{-1})$. A sufficient condition for $F(k) =0$
is the following relation:
\begin{align}
    T(y)+ \tilde T(y) = 0
    .
    \label{eq:funrel-T+tildeT}
\end{align}
Indeed, due to the trace property of $\varphi_0$, we compute:
\begin{align*}
    F(k) =  \varphi_0\brac{
        k^{-2} T(\mathbf y)(\nabla_l k) (\nabla_l k) 
    } &= \varphi_0\brac{
        (\nabla_l k) T(\mathbf y^{-1})((\nabla_l k)k^{-2}) 
    } \\
 &=  \varphi_0\brac{
     k^{-2}  \mathbf y^{-2}T(\mathbf y^{-1})(\nabla_l k) (\nabla_l k) 
 },  
\end{align*} 
the \cref{eq:funrel-T+tildeT} implies $2 F(k) =0$.

The main result  of this section \cref{prop:app-T-functional-EQ}
is  a combinatorial  verification of 
\cref{eq:funrel-T+tildeT} using the Gauss hypergeometric functions  as
building blocks (in particular, making no use of their explicit expressions).  
First of all, we rewrite \cref{eq:b2cal-KDeltak,eq:b2cal-HDeltak} in terms of
hypergeometric functions: 
\begin{align}
    \begin{split}
        K_{\Delta_k}( y ;m) = & \,\,
    -\frac{1}{2} \Gamma \left(  m/2 +1\right)
     \pFq21 \left(  m/2 +1,1,3,z \right) 
     \\ &\,\,
     + \frac{2 \Gamma \left(m/2  +2\right)}{3m}  \pFq21\left(m/2  +2,1,4,z
     \right),
     \end{split}
\label{eq:b2cal-KDelta-hgfun}
    \end{align}
  with $z=1-y$   and
\begin{align}
    \begin{split}
     H_{\Delta_k}( y_1 ,y_2 ;m) = 
& \,\, \frac{1}{6m}    
 2 (m+2) \Gamma \left(m/2+2\right) F_1\left(m/2+2;1,1;4; z_1 , z_2 \right)
 \\ & \,\,
-\frac{1}{6m}\Gamma \left(m/2+3\right)  2 F_1\left(m/2+3;1,1;5;z_1,z_2\right)\\
& \,\,
-\frac{1}{6m}\Gamma \left(m/2+3\right) (1-z_1) F_1\left(m/2+3;2,1;5;z_1,z_2\right).
     \end{split}
\label{eq:b2cal-HDelta-hgfun}
   \end{align}
where $z_1 = 1-y_1$ and $z_2 = 1-y_1 y_2$.
By setting $m=2$, we get 
\begin{lem}
    \label{lem:T-as-hgfun-step-1}
    The function $T(y)$ defined in \cref{eq:funrel-T} is equal to 
    \begin{align}
        \label{eq:app-T}
        T(y) 
        &= 1/6 \left[  1/y + 8 \pFq21(3, 1; 4; 1 - y) - 6 \pFq21( 4,1;5;1
            - y ) 
        \right. \\ & \left. - 
        3 y \pFq21( 4, 2; 5; 1 - y ) \right]
        \nonumber
        .
    \end{align}
\end{lem}
\begin{proof}
   The first term in \eqref{eq:funrel-T} gives rise to $y^{-1}/6$. 
    Special values of $\pFq21$ involved 
    can be easily  calculated via the integral
   representation in \cref{eq:hgeo-defn-pFq21}. 
   
   For the second term, $y_1
   = y_2^{-1}$ implies that $z_2 =0$, which leads to 
the reduction relation \eqref{eq:hgfun-F1toGF21-2} for Appell
$F_1$ functions.    
\end{proof}

\begin{prop}
\label{prop:app-T-functional-EQ}
     The function $T(y)$ defined above \eqref{eq:app-T} can be further 
     simplified to
     \begin{align*}
         T(y) = K(1)\brac{ y^{-1}- \pFq21(3,1;5;1-y) }, 
     \end{align*}
     where $K(1) \defeq K_{\Delta_k}(1;2) = 1/6$.
     Moreover, it satisfies the  functional equation:
\begin{align}
    T(y) + y^{-2} T(y^{-1}) =0.
\label{eq:app-T-functional-EQ}
\end{align}
\end{prop}
\begin{proof}
    According to \cref{eq:hgeofun-Euler}, we have
    $\pFq21( 4,1;5;1-y) = y^{-1} \pFq21( 3,1;5;1 -y) $ and apply one of the
    contiguous relations in \cref{eq:hgeofun-cont-relations}, we see that
    \begin{align*}
        &\,\,  8 \pFq21(3, 1; 4; 1 - y) - 6 \pFq21( 4,1;5;1 - y ) 
- 3 y \pFq21( 4, 2; 5; 1 - y ) \\
=&\,\, -
\pFq21(3,1;5;1-y) .
    \end{align*}
     It remains to check that:
    \begin{align}
        \label{eq:app-T+tildeT-tocheck}
       0= 6\left[  T(y) + y^{-2} T(y^{-1}) \right] &=    
        \frac2y - \pFq21(3,1;5;1-y) 
        \\
        &- y^{-2}\pFq21(3,1;5;1-y^{-1})
        . \nonumber
    \end{align}
    Indeed,
    \begin{align*}
        \pFq21(3,1;5;1-y) &= y \pFq21(2,4;5;1-y), \,\,\,
        \text{by \eqref{eq:hgeofun-Euler},} \\
        y^{-2}\pFq21(3,1;5;1-y^{-1}) &=y \pFq21(3,4;5;1-y),
        \,\,\, 
        \text{by \eqref{eq:hgeofun-Pfaff-1}.}
    \end{align*}
    Therefore \cref{eq:app-T+tildeT-tocheck} above is reduced to:
    \begin{align}
 \label{eq:app-T+tildeT-tocheck-2}
 \pFq21(3,4;5;1-y) + \pFq21(2,4;5;1-y)= 2 y^{-2},
    \end{align}
which follows from one of the contiguous relations in 
\cref{prop:hgeofun-cont-relations}. In fact, take $F = \pFq21(3,4;5;1-s)$ then
$F(a-) = \pFq21(2,4;5;1-y)$, and  then \eqref{eq:app-T+tildeT-tocheck-2}
consists of the third line and the fifth line \cref{eq:hgeofun-cont-relations}:
\begin{align*}
        2F(a-)+ 2F = 4 y F(b+) = 4 y  \pFq21(3,5;5;1-y)
        = 4y^{-2}.
    \end{align*}
    We  need \cref{eq:hgeofun-specialcases} for the last equal sign.
\end{proof}

 \subsection{Connes-Moscovici Type Functional Relations}
On $\T^2_\theta$, a deeper result than the Gauss-Bonnet theorem is the
variational interpretation of the modular Gaussian curvature which leads to
a simple functional relation between $K_{\Delta_k}$ and $H_{\Delta_k}$. 
The variational interpretation 
with respect to the conformal change of metric $\Delta \to  \Delta_\varphi$
extends to all $\T^m_\theta$. 
The variational computation behind \cref{eq:CM-funrel-all-m}
is the main topic of the sequel paper \cite{Liu:2018ab}. Nevertheless,
the theorem below stands as an independent result. It simply asserts that 
the two functions give in 
\cref{eq:checkrels-explict-KDelta,eq:checkrels-explict-HDelta}
are subject to an a priori relation. Most importantly, it strongly validates
all the computations behind both the algebraic expressions of $K_{\Delta_k}$ and
$H_{\Delta_k}$  and the functional relations.
\begin{thm}
    \label{thm:-funrel-all-m}
For all $m \in [2,\infty)$, the spectral functions $K_{\Delta_k}$ and
$H_{\Delta_k}$ defined in
\eqref{eq:checkrels-explict-KDelta} and
\eqref{eq:checkrels-explict-HDelta} respectively satisfy the following
Connes-Moscovici type functional relation:
\begin{align}
    \begin{split}
        &\, \,  - H_{\Delta_k}(y_1,y_2;m)  \\ 
    =&\, \, 
y_1^{-m/2-2} K_{\Delta_k}(z;m)[y_1^{-1}, y_2^{-1}]_z 
- (y_1 y_2)^{-m/2-2} K_{\Delta_k}(z;m)[(y_1 y_2)^{-1}, y_2]_z 
\\
    -&\, \, 
    K_{\Delta_k}(z;m)[y_1 y_2, y_1]_z ,
    \end{split}
    \label{eq:CM-funrel-all-m}
\end{align}
where the divided difference notation is explained in
$\S\ref{subsec:divided-diff-notation}$.
\end{thm}
\begin{rem}
  The continuity with respect to $m$    
hints at the potential of the spectral approach to curvature 
to cope with noncommutative spaces of non-integer dimension.
 In the spectral framework, the notation of dimension is modeled on Weyl's law,
roughly speaking, is proportion to growth rate of the eigenvalues of the Dirac
operator.  Many of such examples have
been constructed, but they have not yet possessed  sophisticated calculus to
initiate detailed computations.
    
\end{rem}
\begin{proof}
    The functional relations listed in \S\ref{sec:diff-and-recursive-relations}
    in are not sufficient to give a combinatorial computation for the
    verification. The assistance of CASs (computer algebra systems) are required
    at the current stage.

    The conceptual part of the computation is given in
    \cref{prop:KDelta-HDelta-alg-expressions} in which we obtain the algebraic 
    expressions of $K_{\Delta_k}$ and $H_{\Delta_k}$ as explicit  functions in $m$. 
    The function \textbf{FullSimplify} in \textsf{Mathematica} is able to
    achieve full cancellation when substituting
    \cref{eq:checkrels-explict-KDelta,eq:checkrels-explict-HDelta} into the two
    sides of \cref{eq:CM-funrel-all-m}.

\end{proof}

%

\begin{prop}
    \label{prop:KDelta-HDelta-alg-expressions}
       For any $m \ge 2$, the modular curvature $R_{\Delta_k}$  in
       \textup{Theorem $\ref{thm:psecal-R_Delta_k}$} involves the following
       spectral functions    
\begin{align}
    \begin{split}
        K_{\Delta_k}(y;m) &= \frac{
  -8 y^{-\frac{m}{2}} \left((m (y-1)-4 y) y^{m/2}+y (m (y-1)+4)\right) \Gamma
  \left(\frac{m}{2}+2\right)      
        }{
        (m-2) m^2 (m+2) (y-1)^3
    } ,
    \end{split}
\label{eq:checkrels-explict-KDelta}
\end{align}
and
\begin{align}
\label{eq:checkrels-explict-HDelta}
    \begin{split}
         &\, \, 
        H_{\Delta_k}(y_1,y_2;m) \\
       =    &\, \,  
        \frac{2}{m}  (y_1-1)^{-2} (t-1)^{-2} (y_1 t-1)^{-3}
    \Gamma(m/2+1)
     \\ &
   \left[ 
   2y_1^{-m/2} (y_1 y_2-1)^3+
   2 (y_2-1)^2 \left(\frac{1}{2} m (y_1-1) (y_1 y_2-1)+y_1 (1-2 y_1) y_2+1\right)
      \right. 
    \\
    & \left. 
        -2 (y_1-1)^2 y_2 (y_1 y_2)^{-\frac{m}{2}} \left(\frac{1}{2} m (y_2-1)
            (y_1 y_2-1)+y_1 y_2^2+y_2-2\right)  
    \right]
. 
    \end{split}
      \end{align}
\end{prop}
\begin{rem}
    To recover $K_{\Delta_k}(y;2)$ and $H_{\Delta_k}(y_1,y_2;2)$, one needs to
    compute the limit $m\to  2$   in  
    \cref{eq:checkrels-explict-KDelta,eq:checkrels-explict-HDelta}.
\end{rem}
\begin{proof}
    Using \textsf{Mathematica}, one computes explicitly the following Gauss
    hypergeometric functions:
    \begin{align*}
        H_{2,1}(z;m) &= \scriptstyle{
\frac{2 (1-z)^{-\frac{m}{2}} \left(-((m-2) z+2) (1-z)^{m/2}-2 z+2\right) \Gamma
\left(\frac{m}{2}+1\right)}{(m-2) m z^2}
}, \\
H_{3,1}(z;m) &= \scriptstyle{
\frac{(1-z)^{-\frac{m}{2}} \left(-((m-2) z (m z+4)+8) (1-z)^{m/2}-8 z+8\right)
\Gamma \left(\frac{m}{2}+2\right)}{m \left(m^2-4\right) z^3}
} ,\\
H_{4,1}(z;m) &= 
\scriptstyle{
\frac{(1-z)^{-\frac{m}{2}} \left(-((m-2) z (m z ((m+2) z+6)+24)+48)
(1-z)^{m/2}-48 (z-1)\right) \Gamma \left(\frac{m}{2}+3\right)}{3 (m-2) m (m+2)
(m+4) z^4}
}.
    \end{align*}
    For the two-variable family, we have the reduction using divided
    difference and differentials:   
    \begin{align*}
        H_{2,1,1}(z_1,z_2;m) &= ( z H_{3,1}(z;m) )[z_1,z_2]_z, \,\,
        H_{3,1,1}(z_1,z_2;m) = ( z H_{4,1}(z;m) )[z_1,z_2]_z \\
        H_{2,2,1}(z_1,z_2;m) &= \partial_{z_1} H_{2,1,1}(z_1,z_2;m).
    \end{align*}
    We remind the reader again $z = 1-y$, $z_1 = 1-y_1$ and $z_2 = 1-y_1y_2$.    
\end{proof}


\bibliographystyle{halpha}

\bibliography{mylib}

\end{document}